\documentclass[12pt,twoside]{amsart}
\usepackage[margin=3cm]{geometry}
\usepackage[colorlinks=false]{hyperref}
\usepackage[english]{babel}
\usepackage{graphicx,titling}
\usepackage{float}
\usepackage{amsmath,amsfonts,amssymb,amsthm}
\usepackage{lipsum}
\usepackage[T1]{fontenc}
\usepackage{fourier}
\usepackage{color}
\usepackage[latin1]{inputenc}
\usepackage{esint}
\usepackage{caption}
\usepackage{ccicons}

\makeatletter
\def\blfootnote{\xdef\@thefnmark{}\@footnotetext}
\makeatother

\newcommand\ccnote{
    \blfootnote{\copyright\,\, Rupert L. Frank and Michael Loss}
    \blfootnote{\ccLogo\, \ccAttribution\,\, Licensed under a \href{https://creativecommons.org/licenses/by/4.0/}{Creative Commons Attribution License (CC-BY)}.}
}

\usepackage[export]{adjustbox}
\numberwithin{equation}{section}
\usepackage{setspace}
\renewcommand{\le}{\leqslant}
\renewcommand{\leq}{\leqslant}
\renewcommand{\ge}{\geqslant}
\renewcommand{\geq}{\geqslant}
\renewcommand{\mathbb}{\varmathbb}
\usepackage{fancyhdr}
\newtheorem{theorem}{Theorem}[section]
\newtheorem{lemma}[theorem]{Lemma}
\newtheorem{corollary}[theorem]{Corollary}
\newtheorem{proposition}[theorem]{Proposition}
\newtheorem{definition}[theorem]{Definition}
\newtheorem{remark}[theorem]{Remark}
\newcommand{\C}{\mathbb{C}}

\newcommand{\const}{\mathrm{const}\ }

\renewcommand{\epsilon}{\varepsilon}

\newcommand{\loc}{{\rm loc}}

\renewcommand{\phi}{\varphi}
\newcommand{\R}{\mathbb{R}}

\newcommand{\Sph}{\mathbb{S}}

\DeclareMathOperator{\re}{Re}

\address{Rupert L. Frank, Mathe\-matisches Institut, Ludwig-Maximilans Universit\"at M\"unchen, The\-resienstr.~39, 80333 M\"unchen, Germany, and Munich Center for Quantum Science and Technology, Schel\-ling\-str.~4, 80799 M\"unchen, Germany, and Mathematics 253-37, Caltech, Pasa\-de\-na, CA 91125, USA}
\email{r.frank@lmu.de}
\medskip
\address{Michael Loss, School of Mathematics, Georgia Institute of
	Technology, Atlanta, GA 30332-0160, USA} 
\email{loss@math.gatech.edu}

\begin{document}

\thispagestyle{empty}

\begin{minipage}{0.28\textwidth}
\begin{figure}[H]
\includegraphics[width=2.5cm,height=2.5cm,left]{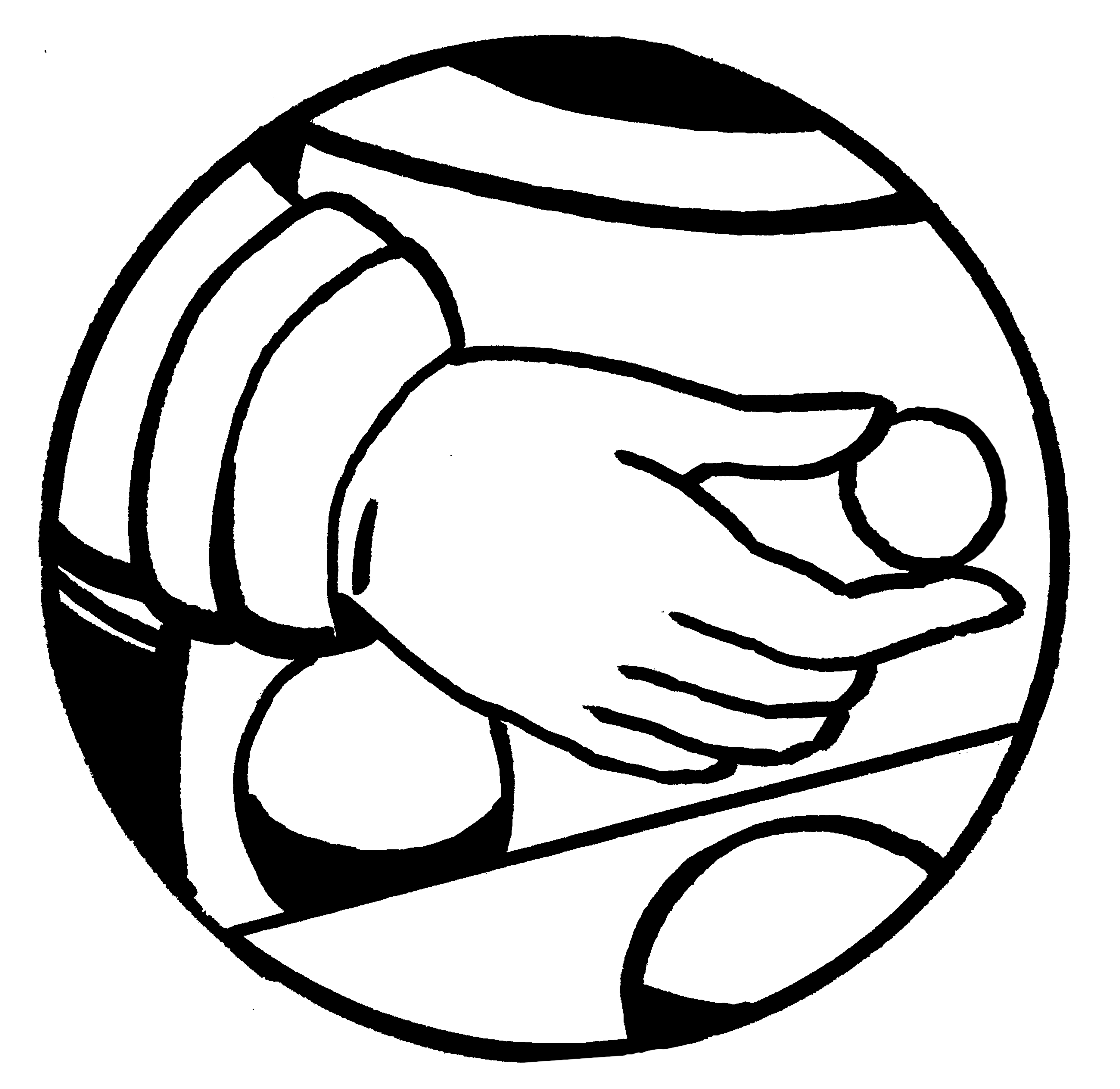}
\end{figure}
\end{minipage}
\begin{minipage}{0.7\textwidth} 
\begin{flushright}
Ars Inveniendi Analytica (2022), Paper No. 1, 31 pp.
\\
DOI 10.15781/rvnn-bp52
\\
ISSN: 2769-8505
\end{flushright}
\end{minipage}

\ccnote

\vspace{1cm}


\begin{center}
\begin{huge}
\textit{Existence of optimizers}

\textit{in a Sobolev inequality for vector fields}

\end{huge}
\end{center}

\vspace{1cm}


\begin{minipage}[t]{.28\textwidth}
\begin{center}
{\large{\bf{Rupert L. Frank}}} \\
\vskip0.15cm
\footnotesize{LMU Munich, MCQST \& California Institute of Technology}
\end{center}
\end{minipage}
\hfill
\noindent
\begin{minipage}[t]{.28\textwidth}
\begin{center}
{\large{\bf{Michael Loss}}} \\
\vskip0.15cm
\footnotesize{Georgia Institute of Technology}
\end{center}
\end{minipage}

\vspace{1cm}


\begin{center}
\noindent \em{Communicated by Jean Dolbeault}
\end{center}
\vspace{1cm}


\noindent \textbf{Abstract.} \textit{We consider the minimization problem corresponding to a Sobolev inequality for vector fields and show that minimizing sequences are relatively compact up to the symmetries of the problem. In particular, there is a minimizer. An ingredient in our proof is a version of the Rellich--Kondrachov compactness theorem for sequences satisfying a nonlinear constraint.}
\vskip0.3cm

\vspace{0.5cm}


\section{Introduction and main result}

Zero modes of the three dimensional Dirac equation play a role in various physical contexts. The have an influence on the Stability of Matter problem (see \cite{FrLiLo,LoYa}) as well as on the understanding  of fermionic functional determinants in three dimensional Quantum Electrodynamics \cite{Fry}. Their occurrence leads in both situations to instabilities and it is therefore of importance to find sharp conditions for their absence. 

A zero mode is a non-trivial solution to the spinor equation
$$
\sigma \cdot (-i\nabla)\psi = \sigma\cdot A \psi
\qquad\text{in}\ \R^3 \,,
$$
where $\psi:\R^3\to\C^2$ is a spinor and $A:\R^3\to\R^3$ is the vector potential with magnetic field $B = \nabla\wedge A$, the curl of $A$. Moreover, $\sigma$ denotes the vector of Pauli matrices
$$
\sigma_1 = \begin{pmatrix}
	0 & 1 \\ 1 & 0
\end{pmatrix},
\qquad
\sigma_2 = \begin{pmatrix}
	0 & -i \\ i & 0
\end{pmatrix},
\qquad
\sigma_3 = \begin{pmatrix}
	1 & 0 \\ 0 & -1
\end{pmatrix}.
$$
The canonical quantity whose size determines whether a solution is possible is $\Vert B\Vert_{3/2}$. If this number is too small. then a zero mode
for this field configuration does not exist. The work in \cite{FrLo} is an attempt to determine the optimal value for this quantity. 
One approach explained in \cite{FrLo} is to consider the chain of inequalities
$$
C_s \Vert \psi \Vert_3  \le \Vert \sigma\cdot (-i\nabla)\psi\Vert_{3/2} = \Vert \sigma\cdot A \psi \Vert_{3/2} = \Vert |A| \psi \Vert_{3/2} \le \Vert A\Vert_3 \Vert \psi \Vert_3 \,,
$$
which shows that necessarily $\Vert A\Vert_3 \ge C_s$, the best constant in the inequality
\begin{equation}
	\label{eq:sobspinor}
	\Vert \sigma\cdot (-i\nabla)\psi\Vert_{3/2} \geq C_s \Vert \psi \Vert_3 \,.
\end{equation}
The validity of this inequality can be shown, for instance, using the Hardy--Littlewood--Sobolev inequality \cite[Theorem 4.3]{LiLo}, although the value of the constant $C_s$ is not known. 

Since the vector potential is gauge dependent, the quantity $\Vert A\Vert_3$ is not the ideal candidate and should be replaced by the $\Vert \nabla \wedge A \Vert_{3/2}$ which is dimensionally the same. Thus, a solution of our problem is furnished by the inequality
\begin{equation} \label{vectorfieldinequality}
	C_v  \Vert A \Vert_3 \le \Vert \nabla\wedge A \Vert_{3/2} 
\end{equation}
valid for all $A$ that satisfy the condition $\nabla\cdot (|A|A)=0$. The reason for this condition on the divergence will be explained below. As a consequence, $C_sC_v \le \Vert B \Vert_{3/2}$ is a necessary condition for the existence of a zero mode.

In this context it is of interest that the spinor
\begin{equation} \label{spinor}
	\psi = \frac{I+i\sigma \cdot x}{(1+|x|^2)^{3/2}}\eta \,,
\end{equation}
where $\eta\in\C^2$ is any constant spinor, is a solution of the Euler-Lagrange equation associated with the inequality \eqref{eq:sobspinor}. 
Likewise, the field
\begin{equation}\label{vectorfield}
	A(x) = \frac{3}{(1+|x|^2)^2}[ (1-|x|^2) w + 2 x\cdot w x + 2 w \wedge x] \,,
\end{equation}
where $w\in \R^3$ is a constant vector, satisfies $\nabla\cdot(|A|A) = 0$ and is a solution of the formal Euler--Lagrange equation associated with the best constant in \eqref{vectorfieldinequality}. It was shown in  \cite{LoYa} that  \eqref{spinor} and \eqref{vectorfield} satisfy the zero mode equation if we require that $w = \langle \eta, \sigma \eta\rangle$. These considerations lend credence to the conjecture that the spinor \eqref{spinor} together with the vector field \eqref{vectorfield} are optimizers for their respective Sobolev inequalities and, moreover, yield the minimal value for 
$\Vert B \Vert_{3/2}$.

This motivates the study of the optimal constants for the two Sobolev inequalities \eqref{eq:sobspinor} and \eqref{vectorfieldinequality} in more detail and, as a first step, we shall prove the existence of optimizers. One should emphasize that with no symmetry results available, it is far from clear how to show that \eqref{spinor} and \eqref{vectorfield} are indeed optimizers of their respective inequalities. We leave this to future investigations.

There is an abundance of results on the existence of minimizers in optimization problems related to Sobolev inequalities. The existence of minimizers in our Sobolev inequality \eqref{eq:sobspinor} for spinor fields follows using ideas from concentration compactness and, while not completely standard, does not pose real problems. 

Perhaps surprisingly, the case of vector fields is significantly harder and the combination of quasilinearity, nonlocality and vectorvaluedness take it outside of the scope of standard methods. While we carry out our analysis only for inequality \eqref{vectorfieldinequality}, we believe that the arguments are more general and can be useful in related problems.

Let us be more specific. We consider
\begin{equation}
	\label{eq:defy}
	\mathcal Y := \left\{ A\in L^3(\R^3,\R^3):\ \nabla\wedge A \in L^{3/2}(\R^3) \right\},
\end{equation}
endowed with the norm
\begin{equation*}
	\| A\|_{\mathcal Y} := \|A\|_3 + \|\nabla\wedge A\|_{3/2} \,.
\end{equation*}
Here $\nabla\wedge\cdot$ denotes the curl operator, understood in distributional sense. We will also make use of the seminorm
\begin{equation}
	\label{eq:defseminorm}
	\VERT A \VERT_3 := \inf_{\phi\in\dot W^{1,3}(\R^3)} \|A-\nabla\phi\|_3 \,,
\end{equation}
where $\dot W^{1,3}(\R^3)$ denotes the space of all real functions $\phi\in L^1_\loc(\R^3)$ with $\nabla\phi\in L^3(\R^3,\R^3)$. Our interest lies in the minimization problem
\begin{equation}
	\label{eq:defs}
	S := \inf_{0\not\equiv A\in\mathcal Y} \frac{\|\nabla\wedge A\|_{3/2}^{3/2}}{\VERT A \VERT_3^{3/2}} \,.
\end{equation}
Note that both numerator and denominator in this minimization problem vanish precisely when $A$ is a gradient field. The validity of the Sobolev inequality, that is, the fact that $S>0$, follows from the Helmholtz decomposition in $L^{3/2}(\R^3,\R^3)$ and the Sobolev inequality in $\dot W^{1,3/2}(\R^3,\R^3)$; see Lemma \ref{spositive}.

Our main result is that the infimum defining $S$ is attained. In fact, we prove the stronger result that all suitably normalized minimizing sequences for $S$ are relatively compact in $\mathcal Y$, up to symmetries. The non-compact symmetries of the minimization problem $S$ are the translation and dilation symmetries, as well as a gauge symmetry, to be discussed momentarily. The problem has also a remarkable conformal symmetry, but this will not play a major role in our arguments. As shown in Lemma \ref{gaugeinf}, the gauge invariance can be broken by choosing for a given $A\in L^3(\R^3,\R^3)$ an $A'\in L^3(\R^3,\R^3)$ with $\nabla\cdot(|A'|A')=0$, where $\nabla\cdot$ denotes the divergence operator, understood in distributional sense. Then $\|A'\|_3 = \VERT A \VERT_3$ and for our minimization problem it is natural to work in this gauge.

The following is our main result.

\begin{theorem}\label{main}
	Let $(A_n)\subset\mathcal Y$ be a minimizing sequence for $S$, satisfying $\nabla\cdot(|A_n|A_n)=0$ and $\| A_n\|_{3}=1$. Then there are $\lambda_n\in(0,\infty)$ and $a_n\in\R^3$ such that, along a subsequence,
	$$
	\lambda_n A_n(\lambda_n(x-a_n))
	$$
	converges in $\mathcal Y$ to a minimizer for $S$. In particular, there is a minimizer for $S$.
\end{theorem}

Using the tools developed in the proof of Theorem \ref{main} we will also be able to prove existence of an optimizer for the problem that motivated our study, namely that of finding the minimal $L^{3/2}$-norm of magnetic fields admitting a zero mode. We define
\begin{align*}
	\Sigma := & \inf\left\{ \|\nabla\wedge A\|_{3/2} :\  A\in \mathcal Y \ \text{and}\ \exists\ 0\not\equiv\psi\in L^3(\R^3,\C^2)\ \text{with}\ \sigma\cdot(-i\nabla-A)\psi = 0 \right\}.
\end{align*}
The equation $\sigma\cdot(-i\nabla-A)\psi = 0$ is understood in distributional sense in $\R^3$ and, as shown in \cite{FrLo}, the requirement $\psi\in L^3$ can be replaced by $\psi\in L^p$ for any $3/2<p<\infty$. We choose the $L^3$ norm here because it appears naturally in the proof.

\begin{theorem}\label{optimalb}
	Let $(A_n)\subset\mathcal Y$ be a minimizing sequence for $\Sigma$, satisfying $\nabla\cdot(|A_n|A_n) =0$. Then there are $\lambda_n\in(0,\infty)$ and $a_n\in\R^3$ such that, along a subsequence,
	$$
	\lambda_n A_n(\lambda_n(x-a_n))
	$$
	converges in $\mathcal Y$ to a minimizer $\tilde A$ for $\Sigma$. In particular, there is a minimizer for $\Sigma$. Moreover, if $(\psi_n)\subset \dot W^{1,3/2}(\R^3,\C^2)$ is a corresponding sequence of zero modes, normalized such that $\|\psi_n\|_3 = 1$, then
	$$
	\lambda_n \psi_n(\lambda_n(x-a_n))
	$$
	converges weakly in $\dot W^{1,3/2}(\R^3,\C^2)$ to $\tilde\psi\not\equiv 0$ and one has $\sigma\cdot(-i\nabla+\tilde A)\tilde\psi = 0$.
\end{theorem}

We return now to the discussion of Theorem \ref{main}. The relative compactness statement in this theorem is the analogue of a theorem of Lions \cite{Lio1} concerning the Sobolev inequality in $\dot W^{1,p}(\R^d)$ for $1<p<d$. In this case, the existence of minimizers and, indeed, the identification of minimizers and the optimal constants are due to Rodemich \cite{Ro}, Aubin \cite{Au2} and Talenti \cite{Ta}.

As we will argue now, in our setting there are significant differences to the scalar setting of $\dot W^{1,p}(\R^d)$ and the difficulties to be overcome are the combination of quasilinearity (coming from the power $3/2$ of the derivative term), nonlocality (coming from the seminorm $\VERT\cdot\VERT_3$) and vector-valuedness (of the objects to be optimized over). While existing methods can deal with any one or two of these difficulties, it is not clear to us whether they can deal with all three of them.

Let us be more specific concerning the various standard methods to prove existence of minimizers. In each case we give examplary references, without being exhaustive. Sym\-metrization-based methods \cite{St,Li} seem to be unable to deal with the vectorvaluedness and, in particular, the curl operator and the divergence constraint. Lieb's method of the missing mass \cite{Li1,BrLi1,BrLi2,FrLiLo} uses typically the Hilbert space nature of the underlying space, or, if not, needs some additional ingredients \cite{Li1,FrLi}. In the framework of Lions' concentration compactness principle \cite{Lio1,Lio2}, it is unclear how to deal with the nonlocal and nonlinear dependence of the term $\nabla\phi$ on $A$ in the definition of the seminorm $\VERT A \VERT_3$. Finally, Yamabe's method of subcritical approximations \cite{Ya,Tr,Au} relies on an $L^\infty$-boundedness result, whose analogue for the curl system is not clear to us.

While not obvious to us, it might very well be possible that one of these methods can be used to prove Theorem \ref{main}. We, however, choose a different approach. In a first step we show that any minimizing sequence has a subsequence which, up to these symmetries, has a nontrivial weak limit, and in a second step we show that this subsequence, in fact, converges strongly. The first step relies on an improved Sobolev inequality, where on the right side $\|\nabla\wedge A\|_{3/2}$ is replaced by its geometric mean with a certain Besov-space norm of $\nabla\wedge A$. In the setting of $\dot H^s(\R^d)$ such inequalities appear in \cite{GeMeOr} and build the basis of profile decompositions \cite{Ge}. The improved Sobolev inequality implies rather directly the existence of a nontrivial weak limit point up to symmetries.

In a second step we slightly alter the minimizing sequence using Ekeland's variational principle in order to deduce that the minimizing sequence satisfies the Euler--Lagrange equation with a small inhomogeneity. From this equation we can deduce that the curls of the elements of the minimizing sequence converge in $L^p_\loc$ for any $p<3/2$. Together with a nonlinear Rellich--Kondrachov theorem, discussed below, this allows us to conclude that the weak limit of the minimizing sequence satisfies the Euler--Lagrange equation and, consequently, that the minimizing sequence converges strongly in $\mathcal Y$. The basic idea behind this second step can be traced back to \cite{GaPe} where the idea is briefly sketched on p.~448 in the case of scalar functions and in an unconstraint variational principle. A related argument appears also in \cite{GaMu}. The vector-valuedness and justifying almost everywhere convergence, however, require several new ingredients compared to the scalar case.

Probably the most significant among these is a nonlinear Rellich--Kondrachov theorem, which says that if $\nabla\wedge A_n\rightharpoonup\nabla\wedge A$ in $L^{3/2}$ and if $\nabla\cdot(|A_n|A_n) = \nabla\cdot(|A|A)=0$, then $A_n\to A$ in $L^q_\loc$ for any $q<3$. In particular, a subsequence converges almost everywhere. The importance here is that the constraint $\nabla\cdot(|A_n|A_n)=0$ is nonlinear and, indeed, if it would be replaced by the linear constraint $\nabla\cdot A_n=0$ the conclusion would follow from the standard form of the Rellich--Kondrachov theorem. Our proof of the nonlinear variant is surprisingly complicated and makes use of deep results by Iwaniec \cite{Iw} on solutions of quasilinear equations. We emphasize that the almost everywhere convergence of minimizing sequences is an ingredient in essentially every proof of existence of minimizers of Sobolev-type inequalities. Therefore, even if a different proof of existence of a minimizer could be found, our nonlinear Rellich--Kondrachov theorem is likely to play a fundamental role in such a proof as well.

As we mentioned before, we explain our technique in the context of one specific inequality, which we find interesting in view of our work in \cite{FrLo}. The methods, however, are much more general and are applicable in a variety of settings. Some immediate extensions concern the generalization of the vector field inequality to arbitrary dimensions $d\geq 3$,
$$
\inf_{\phi\in\dot W^{1,d}} \|A - \nabla\phi\|_d \lesssim \|\nabla\wedge A\|_{d/2} \,.
$$
Also, the case of general exponents $1<p<d$ on the right side and $q=dp/(d-p)$ on the left side should be doable, after some changes in the proof of the nonlinear Rellich--Kondrachov theorem, which currently uses the conformal invariance in the $p=d/2$ case. One could consider this problem in a wider context by considering $k$-forms $\omega$ on $\R^d$ or on any Riemannian manifold. Then the Sobolev-type inequality is of the form
$$
\inf \Vert \omega - d\varphi \Vert_q  \lesssim \Vert d \omega \Vert_p 
$$
where the infimum is taken over $k-1 $ forms $\varphi$ and  $q = dp/(d-p)$. The conformal invariant case
corresponds to $p=\frac{d}{k+1}$ and $q = \frac{d}{k}$. We stay with the case $d=3$ and $k=1$
since the mathematical issues that concern us are of considerable difficulty and these difficulties would only be obfuscated in treating  the general case.

We add that the method works in the case of spinor fields, mentioned at the beginning of this introduction. In this case, however, there is no analogue of the nonlinear divergence constraint and therefore the standard form of the Rellich--Kondrachov theorem suffices. Moreover, in the absence of this constraint Lions's concentration compactness method is applicable and gives the relative compactness of minimizing sequences in a standard way. If one insists on using the method in the present paper, there is only a minor change in the application of Ekeland's theorem because the underlying Banach space is complex. We briefly comment on this in Subsection \ref{sec:ekelandappl}.

\subsection*{Acknowledgement}
$\!\!\!$ Partial support through US National Science Foundation grants DMS-1363432 and DMS-1954995 (R.L.F.) and DMS-1856645 (M.L.), as well as through the German Research Foundation grant EXC-2111-390814868 (R.L.F.) is acknowledged.


\section{Some preliminary results}

In this section we collect some simple results that we will use repeatedly in this paper. We begin with the Helmholtz decomposition in $L^p$. While this is valid for any $1<p<\infty$, we only state it for $p=3/2$, the only case that we will be using.

\begin{lemma}\label{helmholtz}
	Let $A\in \mathcal Y$. Then there are $\tilde A \in \dot W^{1,3/2}(\R^3,\R^3)$ and $\phi\in\dot W^{1,3}(\R^3)$ with
	$$
	A = \tilde A + \nabla\phi
	\qquad\text{and}\qquad
	\nabla \cdot \tilde A = 0 \,.
	$$
	This decomposition is unique, up to adding a constant to $\phi$. Moreover,
	$$
	\| \nabla \otimes \tilde A \|_{3/2} \lesssim \|\nabla \wedge A \|_{3/2} 
	\qquad\text{and}\qquad
	\| \nabla\phi \|_3 \lesssim \| A \|_3	\,.
	$$
\end{lemma}

Here, $\nabla\otimes\tilde A$ denotes the $3\times3$-matrix valued function whose entries are $\partial_j \tilde A_k$, $j,k\in\{1,2,3\}$. We include the proof of this lemma, since we will need the explicit construction later on.

\begin{proof}
	Let
	$$
	\phi(x) := \frac{1}{4\pi} \int_{\R^3} \frac{A(y)\cdot(x-y)}{|x-y|^3} \,dy \,.
	$$
	It follows from an endpoint case of the Hardy--Littlewood--Sobolev inequality \cite{StZy} that $\phi\in BMO(\R^3)$ with $\|\phi\|_{BMO} \lesssim \|A\|_{L^3}$. Moreover, by Calder\'on--Zygmund theory (see, e.g., \cite[Section 9.4]{GiTr}) we know that $\phi$ is weakly differentiable with $\|\nabla\phi\|_{L^3} \lesssim \|A\|_{L^3}$ and, in the sense of distributions,
	$$
	-\Delta\phi = - \nabla\cdot A \,.
	$$
	Let $\tilde A := A - \nabla\phi$ and $B:=\nabla\wedge A$. Then, in the sense of distributions,
	$$
	\nabla\wedge\tilde A = B
	\qquad\text{and}\qquad
	\nabla\cdot\tilde A = 0 \,.
	$$
	Let deduce from this that
	\begin{equation}
		\label{eq:coulombgauge}
		\tilde A(x) = \frac{1}{4\pi} \int_{\R^3} \frac{B(y)\wedge(x-y)}{|x-y|^3}\,dy \,.
	\end{equation}
	We denote the right side by $\tilde A'(x)$. Note that, by the Hardy--Littlewood--Sobolev inequality and the assumption $B\in L^{3/2}$, we have $\tilde A'\in L^3$. Moreover, we have $\nabla\wedge\tilde A'=B$ and $\nabla\cdot\tilde A'=0$, so $\nabla\wedge(\tilde A-\tilde A')=0$ and $\nabla\cdot (\tilde A-\tilde A')=0$. We conclude that $\Delta(\tilde A-\tilde A')=0$. By Weyl's lemma, $\tilde A-\tilde A'$ is smooth and therefore satisfies the mean value property $\tilde A(x) - \tilde A'(x) = 3/(4\pi r^3) \int_{B_r(x)} (\tilde A(y) - \tilde A'(y))\,dy$ for any $a\in\R^3$ and $r>0$. By the above discussion, we have $\tilde A-\tilde A'\in L^3$ and, therefore,
	$$
	\left| \tilde A(x) - \tilde A'(x) \right| \leq \left( \frac{3}{4\pi r^3} \int_{B_r(x)} (\tilde A(y) - \tilde A'(y))^3 \,dy \right)^{1/3} \leq\left( \frac{3}{4\pi r^3} \right)^{1/3} \|\tilde A - \tilde A'\|_{L^3(\R^3,\R^3)}.
	$$
	Letting $r\to\infty$, we conclude that $\tilde A(x) - \tilde A'(x)=0$, and since $x\in\R^3$ is arbitrary we have proved \eqref{eq:coulombgauge}. Note that this argument also gives the uniqueness of the decomposition $A=\tilde A+\nabla\phi$.
	
	Using Calder\'on--Zygmund theory one deduces that $\tilde A$ is weakly differentiable with $\| \nabla \otimes \tilde A \|_{L^{3/2}} \lesssim \|\nabla \wedge A \|_{L^{3/2}}$, as claimed.	
\end{proof}

As an application of the Helmholtz decomposition in $L^{3/2}$, we obtain the validity of the Sobolev inequality. Recall that $S$ was defined in \eqref{eq:defs}.

\begin{lemma}\label{spositive}
	$S>0$
\end{lemma}

\begin{proof}
	Given $A\in\mathcal Y$, let $\tilde A$ and $\phi$ be as in Lemma \ref{helmholtz}. The usual Sobolev inequality in $\dot W^{1,3/2}$ implies
	\begin{align*}
		\|\nabla \wedge A \|_{3/2} & = \|\nabla\wedge\tilde A\|_{3/2} \gtrsim \|\nabla\otimes \tilde A \|_{3/2} \gtrsim \|\tilde A \|_3 \\
		& = \| A-\nabla\phi\|_3 \geq \VERT A \VERT_3 \,.
	\end{align*}
	Thus, $S>0$.
\end{proof}

The next lemma discusses the natural choice of the gauge in our problem. Recall that $\VERT\cdot\VERT_3$ was defined in \eqref{eq:defseminorm}.

\begin{lemma}\label{gaugeinf}
	For any $A\in L^3(\R^3,\R^3)$ there is a $\phi_0\in\dot W^{1,3}(\R^3)$, unique up to an additive constant, such that $\| A - \nabla\phi_0 \|_3 = \VERT A \VERT_3$. Moreover, $\nabla\cdot \left( |A-\nabla\phi_0| (A-\nabla\phi_0) \right)=0$. Conversely, if $A'\in L^3(\R^3,\R^3)$ satisfies $\nabla\cdot (|A'|A')=0$ in $\R^3$, then $\| A' \|_3 = \VERT A' \VERT_3$.
\end{lemma}

\begin{proof}
	The existence of $\phi_0$ follows easily from the fact that $\|\cdot\|_3$ is convex and that
	\begin{equation}
		\label{eq:defm}
		\mathcal M:= \left\{ A\in L^3(\R^3,\R^3):\ \nabla\wedge A = 0 \right\}
	\end{equation}
	is closed in $L^3(\R^3,\R^3)$. The latter follows from standard properties of the distributional curl. Note also that $\nabla\phi\in\mathcal M$ for all $\phi\in\dot W^{1,3}(\R^3)$. Uniqueness of $\nabla\phi_0$ follows from the strict convexity of $\|\cdot\|_3$, and the equation $\nabla\cdot \left( |A-\nabla\phi_0| (A-\nabla\phi_0) \right)=0$ arises as the Euler--Lagrange equation of the minimization problem.
	
	Now assume that $A'\in L^3(\R^3,\R^3)$ satisfies $\nabla\cdot (|A'|A')=0$ in $\R^3$ and let $\phi\in\dot W^{1,3}(\R^3)$. Applying the inequality $f(1)\geq f(0) + f'(0)$ for any convex function on $[0,1]$ to $f(t):= \|A'-t\nabla\phi\|_3$, we obtain $\|A'-\nabla\phi\|_3 \geq \|A'\|_3$.
	Taking the infimum over $\phi$ gives $\VERT A'\VERT_3 \geq \|A'\|_3$, and the reverse inequality is trivial.	
\end{proof}


\section{A nonlinear Rellich--Kondrachov lemma}

In this section we present the technical main result of our paper. To motivate it, we note that if $\tilde A_n, \tilde A\in \mathcal Y$ with $\nabla\wedge \tilde A_n \rightharpoonup \nabla\wedge\tilde A$ in $L^{3/2}(\R^3,\R^3)$ and $\nabla\cdot \tilde A_n=\nabla\cdot \tilde A=0$, then $\tilde A_n\to \tilde A$ in $L^p_\loc(\R^3,\R^3)$ for any $p<3$. This is a consequence of the usual Rellich--Kondrachov lemma for scalar functions, applied to each component of $\tilde A$, since by the Helmholtz decomposition in $L^{3/2}$ (Lemma \ref{helmholtz}) the boundedness of $\nabla\wedge\tilde A_n$ in $L^{3/2}$ together with $\nabla\cdot\tilde A_n=0$ implies boundedness of $\nabla\otimes \tilde A_n$ in $L^{3/2}$.

The following theorem says that the same conclusion remains true if the linear constraint $\nabla\cdot \tilde A_n=0$ is replaced by a nonlinear constraint $\nabla\cdot(|A_n|A_n)=0$. Our proof of this result is rather involved and takes up this and the following section.

\begin{theorem}\label{strongconv}
	Let $A_n, A\in \mathcal Y$ with $\nabla\wedge A_n \rightharpoonup \nabla\wedge A$ in $L^{3/2}(\R^3,\R^3)$ and $\nabla\cdot (|A_n|A_n)=\nabla\cdot(|A|A)=0$. Then $A_n\to A$ in $L^q_\loc(\R^3,\R^3)$ for any $q<3$.
\end{theorem}

While a direct proof of this result on $\R^3$ should be possible, we use the conformal invariance of the relevant norms and prove the corresponding result on $\Sph^3$. We denote by $L^p(\Lambda^k \Sph^3)$ the space of $p$-integrable $k$-forms on $\Sph^3$.

\begin{theorem}\label{strongconvsphere}
	Let $\alpha_n,\alpha\in L^3(\Lambda^1\Sph^3)$ with $d^*(|\alpha_n|\alpha_n)=$ $d^*(|\alpha|\alpha)=0$ and $d\alpha_n \rightharpoonup d\alpha$ in $L^{3/2}(\Lambda^2\Sph^3)$. Then $\alpha_n\to \alpha$ in $L^q(\Lambda^1\Sph^3)$ for any $q<3$.
\end{theorem}

\begin{proof}[Proof of Theorem \ref{strongconv} given Theorem \ref{strongconvsphere}]
	Let $\mathcal S:\R^3\to\Sph^3$ be the (inverse) stereographic projection,
	$$
	\mathcal S_j(x) = \frac{2x_j}{1+x^2} \,,\ j=1,2,3, \,,
	\qquad
	\mathcal S_4(x) = \frac{1-x^2}{1+x^2} \,.
	$$
	To a vector field $A$ on $\R^3$ we associate the vector field $\alpha$ on $\Sph^3$ by
	$$
	A(x) = (D\mathcal S(x))^T \alpha(\mathcal S(x)) \,,
	$$
	where $D\mathcal S$ is the Jacobi matrix of $\mathcal S$. Identifying $\alpha$ with a one-form on $\Sph^3$ via the canonical metric on $\Sph^3$, we see that
	$$
	\inf_{\phi} \int_{\R^3} | A-\nabla\phi|^3\,dx = \inf_\Phi \int_{\Sph^3} |\alpha-d\Phi|^3\,d\omega
	$$
	and
	$$
	\int_{\R^3} |\nabla\wedge A|^{3/2}\,dx = \int_{\Sph^3} |d\alpha|^{3/2}\,d\omega \,,
	$$
	where $\omega$ is the uniform surface measure on the sphere. Similarly, the weak convergence of $\nabla\wedge A_n$ in $L^{3/2}(\R^3,\R^3)$ is equivalent to weak convergence of $d\alpha_n$ in $L^{3/2}(\Lambda^2\Sph^3)$ and the condition $\nabla\cdot(|A_n|A_n)=0$, which arises as the Euler equation of the above minimization problem with respect to $\phi$, is equivalent to $d^*(|\alpha_n|\alpha_n)=0$. It is at this point where we use that the stereographic projection is conformal.
	
	Thus, we are in the situation of Theorem \ref{strongconvsphere} and we conclude that for any $q<3$,
	$$
	\int_{\R^3} |A_n-A|^q \left( \frac{2}{1+x^2}\right)^{3-q} \,dx = \int_{\Sph^3} |\alpha-\alpha_n|^q\,d\omega \to 0 \,.
	$$
	Since the weight on the left side is bounded away from zero on every bounded set, this implies the $L^q_\loc(\R^3,\R^3)$ convergence of $(A_n)$.
\end{proof}

Thus, it remains to prove Theorem \ref{strongconvsphere}. As a preparation we recall the Helmholtz decomposition in $L^{3/2}$ on $\Sph^3$, analogous to that in Lemma \ref{helmholtz} in $\R^3$. Since $d^*\alpha_n$ has integral zero, the solvability of the Poisson problem implies that there is a function $u_n$ on $\Sph^3$ such that $d^*d u_n=d^*\alpha_n$. Thus,
$$
\widetilde{\alpha}_n := \alpha_n - du_n
$$
satisfies
\begin{equation}
	\label{eq:coulomb}
	d\widetilde{\alpha}_n = d\alpha_n
	\qquad\text{and}\qquad
	d^*\widetilde{\alpha}_n = 0 \,.
\end{equation}
Similarly, we define $u$ and $\widetilde{\alpha}$.

We recall the inequality
\begin{equation}
	\label{eq:helmholtzsphere}
	\|\xi\|_{W^{1,3/2}(\Lambda^1\Sph^3)} \leq C \left( \|d\xi\|_{L^{3/2}(\Lambda^2 \Sph^3)} + \|d^*\xi \|_{L^{3/2}(\Sph^3)} \right);
\end{equation}
see, for instance, \cite[Thm.~4.11]{IwScSt} together with the fact that there are no harmonic one-forms on $\Sph^3$.

Inequality \eqref{eq:helmholtzsphere}, applied to $\widetilde{\alpha}_n$, implies that $(\widetilde{\alpha}_n)$ is bounded in $W^{1,3/2}(\Lambda^1\Sph^3)$ and therefore, after passing to a subsequence, we may assume that the sequence $(\widetilde\alpha_n)$ converges weakly in $W^{1,3/2}(\Lambda^1\Sph^3)$. By passing to the limit in \eqref{eq:coulomb} we see that the limit of $\widetilde{\alpha}_n$, which we temporarily denote by $\alpha'$, satisfies.
$$
d\alpha' = d\alpha \,,\qquad d^*\alpha' = 0 \,.
$$
Thus, $d(\alpha'-\widetilde{\alpha})=0$ and $d^*(\alpha'-\widetilde{\alpha})=0$. Since there are no harmonic one-forms, we conclude that $\alpha'=\widetilde{\alpha}$. Thus, $\widetilde{\alpha}_n\rightharpoonup \widetilde{\alpha}$ in $W^{1,3/2}(\Lambda^1\Sph^3)$.

A quick aside: Here we extracted a subsequence, whereas we stated Theorem \ref{strongconvsphere} for the full sequence. To deduce the theorem as stated we note that the proof really shows that any subsequence has a further subsequence such that the conclusion holds, and this proves that the conclusion holds, indeed, along the full sequence.

Next, by the usual Rellich--Kondrachov lemma mentioned at the beginning of this section, $\widetilde\alpha_n\to\widetilde{\alpha}$ in $L^p(\Lambda^1\Sph^3)$ for any $p<3$. Thus, to prove the theorem we need to show that $du_n\to du$ in $L^p(\Lambda^1\Sph^3)$ for any $p<3$. To prove this, we recall the equations satisfied by $\alpha_n$ and $\alpha$, namely,
$$
d^*(|du_n+\widetilde{\alpha}_n|(du_n+\widetilde{\alpha}_n)) = 0
\qquad\text{and}\qquad
d^*(|du+\widetilde{\alpha}|(du+\widetilde{\alpha})) = 0 \,.
$$
We think of this as an equation for $du_n$ for given $\widetilde{\alpha}_n$. The key step in the proof is the following inequality, which says that the solution $u_n$ depends, in some sense, continuously on the data $\widetilde{\alpha}_n$. This is easy in the topology of $L^3$, but rather deep for $L^p$ with $p<3$.

\begin{lemma}\label{contbound}
	There are absolute constants $C<\infty$ and $\epsilon_*>0$ such that, if $0\leq\epsilon\leq\epsilon_*$ and if $\phi_1,\phi_2\in W^{1,3-3\epsilon}(\Sph^3)$ and $\xi_1,\xi_2\in L^{3-3\epsilon}(\Lambda^1\Sph^3)$ satisfy
	$$
	d^*(|d\phi_1-\xi_1|(d\phi_1-\xi_1))=0
	\qquad\text{and}\qquad
	d^*(|d\phi_2-\xi_2|(d\phi_2-\xi_2))=0 \,,
	$$
	then
	\begin{equation}
		\label{eq:contbound}
		\|d\phi_1-d\phi_2\|_{L^{3-3\epsilon}(\Lambda^1 \Sph^3)}^2 \leq C \left( \|\xi_1-\xi_2\|_{L^{3-3\epsilon}(\Lambda^1\Sph^3)} M + \epsilon^2 M^2 \right)
	\end{equation}
	with
	$$
	M := \|\xi_1\|_{L^{3-3\epsilon}(\Lambda^1\Sph^3)} + \|\xi_2\|_{L^{3-3\epsilon}(\Lambda^1\Sph^3)} \,.
	$$
\end{lemma}

This lemma is the analogue of a result in \cite{GrIwSb} concerning the closely related equation $d^*(|d\phi|d\phi)=d\psi$. We defer its proof to the following section.

The `problem' with the bound \eqref{eq:contbound} is the term $\epsilon^2$ on the right side, which only becomes small when $\epsilon\to 0$. However, in our application where $\xi_1=-\widetilde\alpha_n$ and $\xi_2=-\widetilde{\alpha}$, we cannot expect convergence of the first term on the right side of \eqref{eq:contbound} at $\epsilon=0$.

To go around this impasse, we follow \cite{GrIwSb} and deduce from \eqref{eq:contbound} a bound in the grand Lebesgue space $L^{\theta,3)}(\Sph^3)$. This space (which depends on a parameter $\theta> 0$, which only plays a minor role in what follows) strictly contains $L^3(\Sph^3)$. The second ingredient, which is due to \cite{CoKu}, is the observation that the Rellich--Kondrachov theorem remains valid in this space. Combining these two ingredients it will be easy to complete the proof of Theorem \ref{strongconvsphere}.

We now present the details of this argument. For $\theta>0$ we denote by $L^{\theta,3)}(\Sph^3)$ the set of (equivalence classes of) measurable functions $f$ on $\Sph^3$ for which
$$
\| f \|_{L^{\theta,3)}(\Sph^3)} := \sup_{0<\delta\leq 2} \left( \delta^\frac{\theta}{3} |\Sph^3|^{-\frac{1}{3-\delta}} \|f\|_{L^{3-\delta}(\Sph^3)} \right).
$$
is finite. The factor $|\Sph^3|^{-\frac{1}{3-\delta}}$ normalizes the measure on $\Sph^3$, but is not really important.

\begin{corollary}\label{grand}
	There is an absolute constant $C<\infty$ such that if $\phi_1,\phi_2\in W^{1,3}(\Sph^3)$ and $\xi_1,\xi_2\in L^{3}(\Lambda^1\Sph^3)$ satisfy
	$$
	d^*(|d\phi_1-\xi_1|(d\phi_1-\xi_1))=0
	\qquad\text{and}\qquad
	d^*(|d\phi_2-\xi_2|(d\phi_2-\xi_2))=0 \,,
	$$
	then for any $0<\theta\leq 3$,
	\begin{equation}
		\|d\phi_1-d\phi_2\|_{L^{\theta,3)}(\Lambda^1 \Sph^3)} \leq C \, \|\xi_1-\xi_2\|_{L^{\theta,3)}(\Lambda^1\Sph^3)}^{1-\frac\theta3} (M')^{1+\frac\theta 3}
	\end{equation}
	with
	$$
	M' := \|\xi_1\|_{L^{\theta,3)}(\Lambda^1\Sph^3)} + \|\xi_2\|_{L^{\theta,3)}(\Lambda^1\Sph^3)} \,.
	$$
\end{corollary}

\begin{proof}
	We abbreviate $\|\cdot\|_{\theta,3)} := \|\cdot\|_{L^{\theta,3)}(\Sph^3)}$. Since $\|f\|_{L^{3-3\epsilon}(\Sph^3)} \leq (3\epsilon)^{-\frac\theta 3}|\Sph^3|^\frac{1}{3-3\epsilon} \|f\|_{\theta,3)}$, the bound \eqref{eq:contbound} implies
	\begin{equation}
		\|d\phi_1-d\phi_2\|_{L^{3-3\epsilon}(\Sph^3)}^2 \leq C (3\epsilon)^{-\frac{2\theta}3}|\Sph^3|^\frac{2}{3-3\epsilon} \left( \|\xi_1-\xi_2\|_{\theta,3)} M' + \epsilon^2 (M')^2 \right)
	\end{equation}
	for all $0<\epsilon\leq\epsilon_*$ and $\theta>0$. Now given a parameter $0<\delta\leq\min\{\epsilon_*,\frac23\}$, we set
	$$
	\epsilon := \delta\ \frac{\|\xi_1-\xi_2\|_{\theta,3)}^{1/2}}{(M')^{1/2}} \,.
	$$
	Note that, in view of the explicit expression of $M'$, we have $\epsilon\leq\delta$ and therefore,
	\begin{align*}
		|\Sph^3|^{-\frac{2}{3-3\delta}} \|d\phi_1-d\phi_2\|_{L^{3-3\delta}(\Sph^3)}^2
		& \leq |\Sph^3|^{-\frac{2}{3-3\epsilon}} \|d\phi_1-d\phi_2\|_{L^{3-3\epsilon}(\Sph^3)}^2 \\
		& \leq C (3\epsilon)^{-\frac{2\theta}3} \left( \|\xi_1-\xi_2\|_{\theta,3)} M' + \epsilon^2 (M')^2 \right) \\
		& = C (3\delta)^{-\frac{2\theta}3} \|\xi_1-\xi_2\|_{\theta,3)}^{1-\frac{\theta}{3}} (M')^{1+\frac{\theta}{3}} \left(1+\delta^2\right) \\
		& \leq C'(3\delta)^{-\frac{2\theta}3} \|\xi_1-\xi_2\|_{\theta,3)}^{1-\frac{\theta}{3}} (M')^{1+\frac{\theta}{3}}
	\end{align*}
	with $C'=C (1+ \min\{\epsilon_*^2,\frac49\})$. Moreover, in case $\epsilon_*<\frac23$, we bound for $\epsilon_*<\delta\leq\frac23$,
	\begin{align*}
		|\Sph^3|^{-\frac{2}{3-3\delta}} \|d\phi_1-d\phi_2\|_{L^{3-3\delta}(\Sph^3)}^2
		& \leq |\Sph^3|^{-\frac{2}{3-3\epsilon_*}} \|d\phi_1-d\phi_2\|_{L^{3-3\epsilon_*}(\Sph^3)}^2 \\
		& \leq C'(3\epsilon_*)^{-\frac{2\theta}3} \|\xi_1-\xi_2\|_{\theta,3)}^{1-\frac{\theta}{3}} (M')^{1+\frac{\theta}{3}} \\
		& \leq C'(3\epsilon_*/2)^{-\frac{2\theta}3} (3\delta)^{-\frac{2\theta}3} \|\xi_1-\xi_2\|_{\theta,3)}^{1-\frac{\theta}{3}} (M')^{1+\frac{\theta}{3}} \,.
	\end{align*}
	
	To summarize, we have for all $0<\delta\leq\frac23$,
	\begin{align*}
		|\Sph^3|^{-\frac{2}{3-3\delta}} \|d\phi_1-d\phi_2\|_{L^{3-3\delta}(\Sph^3)}^2
		& \leq C_\theta (3\delta)^{-\frac{2\theta}3} \|\xi_1-\xi_2\|_{\theta,3)}^{1-\frac{\theta}{3}} (M')^{1+\frac{\theta}{3}}
	\end{align*}
	with $C_\theta :=C' \max\{1,(3\epsilon_*/2)^{-\frac{2\theta}3}\}$. This implies
	$$
	\|d\phi_1-d\phi_2\|_{\theta,3)}^2 \leq C_\theta \|\xi_1-\xi_2\|_{\theta,3)}^{1-\frac{\theta}{3}} (M')^{1+\frac{\theta}{3}} \,.
	$$
	Since $\theta\leq 3$, we have $C_\theta\leq C_3$ and we obtain the claimed bound.
\end{proof}

As we mentioned already, the second ingredient in the proof of Theorem \ref{strongconvsphere} is a version of the Rellich--Kondra\-chov lemma in grand Lebesgue spaces. This appears as \cite{CoKu}, but we give a self-contained and elementary proof.

\begin{lemma}\label{rellichgrand}
	Assume that $v_n\rightharpoonup 0$ in $W^{1,3/2}(\Sph^3)$. Then $v_n\to 0$ in $L^{\theta,3)}(\Sph^3)$ for any $\theta>0$.
\end{lemma}

\begin{proof}
	For any $\delta_0>0$, we bound, using H\"older's inequality,
	\begin{align*}
		\|v_n\|_{L^{\theta,3)}(\Sph^3)} & \leq \sup_{0<\delta\leq\delta_0} 
		\left( \delta^\frac{\theta}{3} |\Sph^3|^{-\frac{1}{3-\delta}} \|v_n\|_{3-\delta} \right)
		+ \sup_{\delta_0\leq\delta\leq 2} \left( \delta^\frac{\theta}{3} |\Sph^3|^{-\frac{1}{3-\delta}} \|v_n \|_{3-\delta} \right) \\
		& \leq \delta_0^\frac{\theta}{3} |\Sph^3|^{-\frac{1}{3}} \|v_n\|_{3} + 2^\frac{\theta}{3} |\Sph^3|^{-\frac{1}{3-\delta_0}} \|v_n \|_{3-\delta_0} \,.
	\end{align*}
	By the ordinary Rellich--Kondrachov lemma, we have $v_n\to 0$ in $L^{3-\delta_0}(\Sph^3)$, so
	$$
	\limsup_{n\to\infty} \|v_n\|_{L^{\theta,3)}(\Sph^3)} \leq \delta_0^\frac{\theta}{3} |\Sph^3|^{-\frac{1}{3}} \limsup_{n\to\infty} \|v_n\|_{3} \,.
	$$
	Since $(v_n)$ is bounded in $L^3(\Sph^3)$ by Sobolev and since $\delta_0>0$ can be chosen arbitrarily small, we obtain the assertion.
\end{proof}

We are now in position to complete the proof of Theorem \ref{strongconvsphere}. Indeed, Lemma \ref{rellichgrand} implies that $\widetilde{\alpha}_n\to\widetilde{\alpha}$ in $L^{\theta,3)}(\Sph^3)$ for any $\theta>0$. Thus, by Corollary \ref{grand}, $d u_n\to d u$ in $L^{\theta,3)}(\Lambda^1 \Sph^3)$ for any $\theta>0$. Since $\|f\|_{L^q(\Sph^3)}\leq C_{q,\theta} \|f\|_{L^{\theta,3)}(\Sph^3)}$ for any $q<3$ and $\theta>0$, we conclude that $d u_n\to d u$ in $L^q(\Lambda^1 \Sph^3)$ for any $q<3$. Since $\alpha_n=\widetilde{\alpha}_n+du_n$ and $\widetilde{\alpha}_n\to\widetilde{\alpha}$ in $L^q(\Sph^3)$ for any $q<3$, this proves the assertion.
\qed


\section{Nonlinear Helmholtz decomposition}

Our goal in this section is to prove Lemma \ref{contbound}. The key ingredient is a nonlinear version of the Helmholtz decomposition due to Iwaniec \cite{Iw}. A simplified proof of an improved result appears in \cite{IwSb} in the case of Euclidean space and the result for Riemannian manifolds is in \cite[Proof of Thm.~8.8]{IwScSt}. We only state the special case of the result that we need.

\begin{theorem}\label{nonlinearhelm}
	There is an absolute constant $C<\infty$ such that for any $0\leq\epsilon<\frac13$ and any $\phi\in W^{1,3-3\epsilon}(\Sph^3)$ there are $\psi\in W^{1,\frac{3-3\epsilon}{1-3\epsilon}}(\Sph^3)$ and $\gamma\in L^\frac{3-3\epsilon}{1-3\epsilon}(\Lambda^1\Sph^3)$ such that
	$$
	|d\phi|^{-3\epsilon} d\phi = d\psi  + \gamma \,,
	\qquad d^*\gamma=0
	$$
	and
	$$
	\|\gamma\|_{L^\frac{3-3\epsilon}{1-3\epsilon}(\Lambda^1\Sph^3)} \leq C\, \epsilon\, \|d\phi\|_{L^{3-3\epsilon}(\Sph^3)}^{1-3\epsilon} \,.
	$$
\end{theorem}

With this theorem at our disposal, we now turn to the proof of Lemma \ref{contbound}. As we already mentioned, our proof is analogous to the proof of a similar result for a related equation in \cite{GrIwSb}.

\begin{proof}[Proof of Lemma \ref{contbound}]
	According to Theorem \ref{nonlinearhelm}, for any $0\leq\epsilon<\frac13$ there are $\psi\in W^{1,\frac{3-3\epsilon}{1-3\epsilon}}(\Sph^3)$ and $\gamma\in L^\frac{3-3\epsilon}{1-3\epsilon}(\Lambda^1\Sph^3)$ such that
	$$
	|d\phi_1-d\phi_2|^{-3\epsilon} \left(d\phi_1-d\phi_2\right) = d\psi  + \gamma \,,
	\qquad d^*\gamma=0
	$$
	and, with the obvious abbreviation for the norm,
	$$
	\|\gamma\|_{\frac{3-3\epsilon}{1-3\epsilon}} \leq C\, \epsilon\, \|d\phi_1-d\phi_2\|_{3-3\epsilon}^{1-3\epsilon} \,.
	$$
	Testing the equations for $\phi_1$ and $\phi_2$ against $\psi$ and subtracting them from each other, we get
	\begin{align}\label{eq:eqhelm}
		& \int_{\Sph^3} \langle |d\phi_1-\xi_1|(d\phi_1-\xi_1)-|d\phi_2-\xi_2|(d\phi_2-\xi_2), |d\phi_1-d\phi_2|^{-3\epsilon} (d\phi_1-d\phi_2)\rangle\, d\omega \notag \\
		& \qquad = \int_{\Sph^3} \langle |d\phi_1-\xi_1|(d\phi_1-\xi_1)-|d\phi_2-\xi_2|(d\phi_2-\xi_2),\gamma\rangle\, d\omega \,.
	\end{align}
	We will bound the right side from above and the left side from below.
	
	Using
	\begin{equation}
		\label{eq:elementary0}
		\left| |x|x - |y|y\right|\leq \left(|x|+|y|\right)|x-y|
		\qquad\text{for all}\ x,y\in\R^n \,,
	\end{equation}
	(which can be seen by adding and subtracting $|x|y$ from the vector on the left side) we get
	\begin{align*}
		& \left| \langle |d\phi_1-\xi_1|(d\phi_1-\xi_1)-|d\phi_2-\xi_2|(d\phi_2-\xi_2),\gamma\rangle \right| \\
		& \qquad \leq \left( |d\phi_1-\xi_1| + |d\phi_2-\xi_2| \right) |d\phi_1-\xi_1-d\phi_2+\xi_2| \, |\gamma| \\
		& \qquad \leq \left( |d\phi_1| + |d\phi_2| + |\xi_1| + |\xi_2| \right)\left( |d\phi_1-d\phi_2| + |\xi_1-\xi_2|\right) |\gamma| \,.
	\end{align*}
	By H\"older's inequality,
	\begin{align*}
		& \left| \int_{\Sph^3} \langle |d\phi_1-\xi_1|(d\phi_1-\xi_1)-|d\phi_2-\xi_2|(d\phi_2-\xi_2),\gamma\rangle\, d\omega \right| \\
		& \qquad \leq \mu \left( \|  d\phi_1-d\phi_2\|_{3-3\epsilon} + \|\xi_1-\xi_2\|_{3-3\epsilon} \right) \|\gamma\|_{\frac{3-3\epsilon}{1-3\epsilon}} \\
		& \qquad \leq C \mu \epsilon \left( \|  d\phi_1-d\phi_2\|_{3-3\epsilon} + \|\xi_1-\xi_2\|_{3-3\epsilon} \right) \|d\phi_1-d\phi_2\|_{3-3\epsilon}^{1-3\epsilon}
	\end{align*}
	with
	$$
	\mu:= \|d\phi_1\|_{3-3\epsilon} + \|d\phi_2\|_{3-3\epsilon} + \|\xi_1\|_{3-3\epsilon} + \|\xi_2\|_{3-3\epsilon} \,.
	$$
	
	We now turn to the left side in \eqref{eq:eqhelm}. It is elementary to see that
	\begin{align}\label{eq:elementaryhelm}
		& \langle |x-a|(x-a)-|y-b|(y-b),(x-y)\rangle \notag \\
		& \geq \tfrac12 |x-y|^3 - \left( |x|+|y|+|a|+|b|\right)|a-b||x-y| \notag \\
		& \qquad \text{for all}\ x,y,a,b\in\R^n \,.
	\end{align}
	We provide the details at the end of this proof. It follows from this inequality that
	\begin{align*}
		&  \langle |d\phi_1-\xi_1|(d\phi_1-\xi_1)-|d\phi_2-\xi_2|(d\phi_2-\xi_2), |d\phi_1-d\phi_2|^{-3\epsilon} (d\phi_1-d\phi_2)\rangle \\
		& \geq \tfrac12 |d\phi_1-d\phi_2|^{3-3\epsilon} - \left( |d\phi_1|+|d\phi_2|+|\xi_1|+|\xi_2|\right) |\xi_1-\xi_2|\, |d\phi_1-d\phi_2|^{1-3\epsilon} \,.
	\end{align*} 
	Therefore, by H\"older's inequality,
	\begin{align*}
		&  \int_{\Sph^3} \langle |d\phi_1-\xi_1|(d\phi_1-\xi_1)-|d\phi_2-\xi_2|(d\phi_2-\xi_2), |d\phi_1-d\phi_2|^{-3\epsilon} (d\phi_1-d\phi_2)\rangle\, d\omega \\
		& \geq \tfrac12 \|d\phi_1-d\phi_2\|_{3-3\epsilon}^{3-3\epsilon} - \mu \, \|\xi_1-\xi_2\|_{3-3\epsilon}\, \|d\phi_1-d\phi_2\|_{3-3\epsilon}^{1-3\epsilon} \,.
	\end{align*} 
	
	Combining the bounds on both sides of \eqref{eq:eqhelm}, we obtain
	\begin{align*}
		\tfrac12 \|d\phi_1-d\phi_2\|_{3-3\epsilon}^{3-3\epsilon} & \leq C \mu \epsilon \left( \|  d\phi_1-d\phi_2\|_{3-3\epsilon} + \|\xi_1-\xi_2\|_{3-3\epsilon} \right) \|d\phi_1-d\phi_2\|_{3-3\epsilon}^{1-3\epsilon} \\
		& \quad + \mu \, \|\xi_1-\xi_2\|_{3-3\epsilon}\, \|d\phi_1-d\phi_2\|_{3-3\epsilon}^{1-3\epsilon} \,,
	\end{align*}
	which is the same as
	\begin{equation}
		\label{eq:proofcontbound}
		\tfrac12 \|d\phi_1-d\phi_2\|_{3-3\epsilon}^2 \leq C \mu \epsilon \left( \|  d\phi_1-d\phi_2\|_{3-3\epsilon} + \|\xi_1-\xi_2\|_{3-3\epsilon} \right) + \mu \, \|\xi_1-\xi_2\|_{3-3\epsilon}\,.
	\end{equation}
	Absorbing the term $\|  d\phi_1-d\phi_2\|_{3-3\epsilon}$ on the right side into the left side gives
	\begin{equation}
		\label{eq:proofcontbound2}
		\|d\phi_1-d\phi_2\|_{3-3\epsilon}^2 \leq C' \left( \|\xi_1-\xi_2\|_{3-3\epsilon}\mu + \epsilon^2 \mu^2 \right)
	\end{equation}
	with an absolute constant $C'<\infty$.
	
	This is almost the claimed bound, except that we need to replace $\mu$ by $M$. This is where the restriction on $\epsilon$ comes in. We return to \eqref{eq:proofcontbound} in the special case where $d\phi_2=\xi_2=0$, that is,
	$$
	\tfrac12 \|d\phi_1\|_{3-3\epsilon}^2 \leq C\, \epsilon \left( \|  d\phi_1\|_{3-3\epsilon} + \|\xi_1\|_{3-3\epsilon} \right)^2 + \left( \|  d\phi_1\|_{3-3\epsilon} + \|\xi_1\|_{3-3\epsilon} \right) \|\xi_1\|_{3-3\epsilon} \,.
	$$
	We restrict ourselves to $\epsilon\leq 1/(4C)=:\epsilon_*$. Then the term $\|d\phi_1\|_{3-3\epsilon}^2$ on the right side can be absorbed into the left side. Of course, all the factors $\|d\phi_1\|_{3-3\epsilon}$ on the right side can be absorbed as well. In this way, we finally arrive at
	$$
	\|d\phi_1\|_{3-3\epsilon} \leq C''\, \|\xi_1\|_{3-3\epsilon}
	\qquad\text{for all}\ 0\leq\epsilon\leq\epsilon_*
	$$
	with an absolute constant $C''<\infty$. This, together with a similar bound for $d\phi_2$, gives $\mu\leq (1+C'')M$, which, when inserted into \eqref{eq:proofcontbound2}, completes the proof.	
\end{proof}

\begin{proof}[Proof of \eqref{eq:elementaryhelm}]
	Let $c:=(a+b)/2$ and write
	\begin{align*}
		& \langle |x-a|(x-a)-|y-b|(y-b),(x-y)\rangle = I_0 + I_1 + I_2
	\end{align*}
	with
	\begin{align*}
		I_0 & := \langle |x-c|(x-c)-|y-c|(y-c),(x-y)\rangle \,, \\
		I_1 & := \langle |x-a|(x-a)-|x-c|(x-c),(x-y)\rangle \,, \\
		I_2 & := \langle |y-c|(y-c)-|y-b|(y-b),(x-y)\rangle \,.
	\end{align*}
	
	To bound $I_0$ from below, we note that
	\begin{equation}
		\label{eq:elementaryhelm1}
		\langle |X|X-|Y|Y,(X-Y)\rangle = \frac12 \left( (|X|+|X|)|X-Y|^2 + \mathcal R\right)
	\end{equation}
	with
	\begin{equation}
		\label{eq:elementaryhelm2}
		\mathcal R := |X|^3 + |Y|^3 - |X||Y|(|X|+|Y|) \geq 0 \,.
	\end{equation}
	This follows from $|X|^2 |Y|\leq \frac23|X|^3 + \frac13|Y|^3$ and $|X||Y|^2 \leq \frac13|X|^3+\frac23|Y|^3$. Applying \eqref{eq:elementaryhelm1}, \eqref{eq:elementaryhelm2} with $X=x-c$ and $Y=y-c$ gives
	$$
	I_0 \geq \frac12 \left( |x-c|+|y-c| \right)|x-y|^2 \geq \frac12 |x-y|^3 \,.
	$$
	
	To bound $I_1$ from above, we bound, using \eqref{eq:elementary0},
	\begin{align*}
		\left| I_1 \right| & \leq \left| |x-a|(x-a)-|x-c|(x-c) \right| \left|x-y\right| \\
		& \leq \left( |x-a| + |x-c| \right) \left|a-c\right| \left|x-y\right| \\
		& \leq \left( 2|x| + |a| + |c| \right) \left|a-c\right| \left|x-y\right| \\
		& \leq \left( |x| + \tfrac34 |a| + \tfrac14|b| \right)  \left|a-b\right| \left|x-y\right| \,.
	\end{align*}
	This and the corresponding bound on $I_2$ give
	$$
	\left| I_1 + I_2 \right| \leq \left( |x| + |y| + |a| + |b| \right)  \left|a-b\right| \left|x-y\right| \,,
	$$
	which yields the claimed bound.
\end{proof}


\section{Another Rellich--Kondrachov lemma}

This section is a short digression and its content is not needed for the proof of Theorem \ref{main}. We present a different Rellich--Kondrachov lemma for vector fields which might prove useful in other applications.

We need to introduce a gauge-invariant local $L^2$ (semi)norm. Let $\Omega\subset\R^3$ be an open set and define, for $A\in L^2(\Omega,\R^3)$,
$$
\VERT A\VERT_{2\Omega} := \inf_{\phi\in \dot H^1(\Omega)} \|A-\nabla\phi\|_{L^2(\Omega,\R^3)} \,.
$$
Here, $\dot H^1(\Omega)$ denotes the space of all real functions $\phi\in L^1_\loc(\Omega)$ such that $\nabla\phi\in L^2(\R^3,\R^3)$.

The main result of this section is as follows.

\begin{proposition}\label{rellich}
	Suppose that $\nabla\wedge A_n \rightharpoonup \nabla\wedge A$ in $L^{3/2}(\R^3,\R^3)$. Then for any open set $\Omega\subset\R^3$ of finite measure, $\VERT A _n - A\VERT_{2\Omega} \to 0$.
\end{proposition}

For the proof of this proposition, we express $\VERT A\VERT_{2\Omega}$ by duality. For $B\in L^2(\Omega,\R^3)$ we say that
$$
\nabla\cdot B = 0
\quad\text{in}\ \Omega
\qquad\text{and}\qquad
\nu\cdot B = 0
\quad\text{on}\ \partial\Omega
$$
if $\int_\Omega \nabla\phi\cdot B\,dx =0$ for any $\phi\in \dot H^1(\Omega)$. Clearly, if $B\in C^1(\Omega)\cap C(\overline\Omega)$ and $\partial\Omega$ is Lipschitz, this definition coincides with the classical one.

\begin{lemma}\label{duality}
	For any $A\in L^2(\Omega,\R^3)$,
	\begin{equation}
		\label{eq:duality}
		\VERT A\VERT_{2\Omega} = \sup\left\{ \int_\Omega A\cdot B\,dx :\ \|B\|_{L^2(\Omega,\R^3)}\leq 1 \,,\ \nabla\cdot B=0\ \text{in}\ \Omega \,,\ \nu\cdot B=0\ \text{on}\ \partial\Omega \right\}.
	\end{equation}	
\end{lemma}

\begin{proof}
	For any $B$ as on the right side of \eqref{eq:duality} and any $\phi\in H^1(\Omega)$, we have
	$$
	\int_\Omega A\cdot B\,dx = \int_\Omega (A-\nabla\phi)\cdot B\,dx \leq \|A-\nabla\phi\|_{L^2(\Omega,\R^3)} \,.
	$$
	Taking the infimum over $\phi$ and the supremum over $B$ we obtain $\geq$ in \eqref{eq:duality}.
	
	Conversely, as in Lemma \ref{gaugeinf}, there is a $\phi_*\in H^1(\Omega)$ such that
	$$
	\|A-\nabla\phi_*\|_{L^2(\Omega,\R^3)} = \VERT A\VERT_{2\Omega}.
	$$
	The Euler--Lagrange equation corresponding to this minimization problem is
	$$
	\int_\Omega \nabla\phi\cdot(A-\nabla\phi_*)\,dx =0
	\qquad\text{for all}\ \phi\in \dot H^1(\Omega) \,,
	$$
	that is, $B_*:=A-\nabla\phi_*$ satisfies $\nabla\cdot B_*=0$ in $\Omega$ and $\nu\cdot B_*=0$ on $\partial\Omega$. If $B_*\equiv 0$ then $\VERT A\VERT_{2\Omega}=0$ and $\leq$ in \eqref{eq:duality} holds trivially. Otherwise, $B_*/\|B_*\|_{L^2(\Omega)}$ is an admissible candidate for the right side in \eqref{eq:duality} and we have
	$$
	\int_\Omega A\cdot \frac{B_*}{\|B_*\|_{L^2(\Omega,\R^3)}}\,dx = \int_\Omega (A-\nabla\phi_*)\cdot \frac{B_*}{\|B_*\|_{L^2(\Omega,\R^3)}}\,dx = \VERT A\VERT_{2\Omega} \,. 
	$$
	This proves $\leq$ in \eqref{eq:duality}.
\end{proof}

\begin{lemma}\label{smoothing}
	Let $0\leq \eta\in C^1(\R^3)$ with sufficiently fast decay and $\int_{\R^3} \eta\,dx =1$ and set $\eta_\varepsilon(x) := \varepsilon^{-3}
	\eta(\frac x\varepsilon)$. Then
	$$
	\VERT A  - \eta_\varepsilon \star A \VERT_{2\R^3} \leq C_\eta \, \sqrt{\varepsilon}\, \Vert \nabla\wedge A\Vert _{3/2} \,.
	$$
\end{lemma}

\begin{proof}
	We use Lemma \ref{duality} with $\Omega=\R^3$ and consider $B\in L^2(\R^3,\R^3)$ with $\|B\|_2\leq 1$ and $\nabla\cdot B=0$ in $\R^3$. By Plancherel (with the normalization of the Fourier transform as, for instance, in \cite{LiLo}), we have
	$$
	\int_{\R^3} (A  - \eta_\varepsilon \star A) \cdot B dx = \int_{\R^3} (1- \widehat \eta_\varepsilon (k)) \overline {\widehat A(k)} \cdot \widehat B(k) \,dk \,.
	$$
	Since $k \cdot  \widehat B(k) = 0$ we can write 
	$$
	\overline {\widehat A(k)} \cdot \widehat B(k) =\frac{ (k \wedge \overline {\widehat A(k)})\cdot (k \wedge \widehat B(k))}{|k|^2} \,,
	$$
	and hence
	$$
	\int_{\R^3} (A  - \eta_\varepsilon \star A) \cdot B \,dx = \int_{\R^3} \frac{(1- \widehat \eta_\varepsilon(k))}{|k|^{1/2}} \frac{ (k \wedge \overline {\widehat A(k)})}{|k|^{1/2}} \cdot \frac{(k \wedge \widehat B(k))}{|k|} dk \,,
	$$
	which is bounded above by
	$$
	C\sqrt{\varepsilon} \left( \int_{\R^3} \Big | \frac{ (k \wedge \overline {\widehat A(k)})}{|k|^{1/2}}  \Big |^2 dk\right)^{1/2} \,.
	$$
	Here we used the fact that, because of the sufficiently fast decay of $\eta$, we have $\sup |\xi|^{-1/2}|1-\widehat\eta(\xi)|<\infty$. The square of the last factor is, by Plancherel and Hardy--Littlewood--Sobolev,
	$$
	C \left( \nabla\wedge A, \frac{1}{|x|^2} \star \nabla\wedge A\right) \le C \Vert \nabla\wedge A \Vert_{3/2}^2 \,.
	$$
	This is the claimed inequality.
\end{proof}

Finally, we are in position to prove the main result of this section.

\begin{proof}[Proof of Proposition \ref{rellich}]
	We will be using the following two properties of the seminorm $\VERT \cdot\VERT_{2\Omega}$. First, we have
	$$
	\VERT A_1 + A_2\VERT_{2\Omega} \leq \VERT A_1\VERT_{2\Omega} + \VERT A_2\VERT_{2\Omega} \,.
	$$
	This follows from the definition of the seminorm. Second, we have for $A\in L^2(\R^3,\R^3)$,
	$$
	\VERT A\VERT_{2\Omega} \leq\VERT A\VERT_{2\R^3} \,.
	$$
	This follows from the duality lemma \ref{duality} and the fact that if $B\in L^2(\Omega,\R^3)$ satisfies $\nabla\cdot B=0$ in $\Omega$ and $\nu\cdot B=0$ on $\partial\Omega$, then its extension $\tilde B$ by zero to $\R^3$ satisfies $\nabla\cdot \tilde B=0$ in $\R^3$.
	
	Using these two facts, together with Lemma \ref{smoothing} we can bound
	\begin{align*}
		\VERT A_n - A \VERT_{2\Omega} & \leq \VERT A_n - \eta_\varepsilon \star A_n \VERT_{2\Omega}+ \VERT \eta_\varepsilon \star A_n - \eta_\varepsilon \star A\VERT_{2\Omega}+ \VERT \eta_\varepsilon \star A - A \VERT_{2\Omega} \\
		& \leq \VERT A_n - \eta_\varepsilon \star A_n \VERT_{2\R^3}+ \VERT \eta_\varepsilon \star A_n - \eta_\varepsilon \star A\VERT_{2\Omega}+ \VERT \eta_\varepsilon \star A - A \VERT_{2\R^3} \\
		& \leq 2C\sqrt {\varepsilon} +  \VERT \eta_\varepsilon \star A_n - \eta_\varepsilon \star A \VERT_{2\Omega} \,.
	\end{align*}
	
	We need to show that every subsequence has a further subsequence along which, for every fixed $\epsilon>0$, $\VERT \eta_\varepsilon \star A_n - \eta_\varepsilon \star A\VERT_{2\Omega}\to 0$. Since $\nabla\wedge A_n \rightharpoonup \nabla\wedge A$ in $L^{3/2}(\R^3,\R^3)$ we have, by the Sobolev inequality (Lemma \ref{spositive}), 
	\begin{equation}
		\label{eq:weakconvmod}
		\int_{\R^3} A_n\cdot B \,dx\to \int_{\R^3} A\cdot B\,dx
		\qquad\text{for any}\ B\in L^{3/2}(\R^3,\R^3) \ \text{with}\ \nabla\cdot B=0 \,. 
	\end{equation}
	Moreover, since $\VERT A_n-A \VERT_{3}$ is bounded, there is a sequence $\Phi_n$ such that $\|A_n-A-\nabla\Phi_n\|_{3}$ is bounded. Now for the given subsequence, there is a further subsequence along which $A_n-A-\nabla\Phi_n\rightharpoonup F$ in $L^3(\R^3,\R^3)$. It follows from \eqref{eq:weakconvmod} that $\int_{\R^3} F\cdot B\,dx=0$ for all $B\in L^{3/2}(\R^3,\R^3)$ with $\nabla\cdot B=0$, that is, $F=\nabla\Phi$. Thus, $A_n-A-\nabla\phi_n\rightharpoonup 0$ in $L^3(\R^3,\R^3)$ with $\phi_n := \Phi_n-\Phi$. Thus, $\eta_\varepsilon \star A_n - \eta_\varepsilon \star A- \eta_\varepsilon \star\nabla\phi_n$ converges pointwise to zero and is bounded uniformly in $n$. Thus, by dominated convergence, $\eta_\varepsilon \star A_n - \eta_\varepsilon \star A- \eta_\varepsilon \star\nabla\phi_n\to 0$ in $L^2(\Omega)$ and so,
	$$
	\VERT \eta_\varepsilon \star A_n - \eta_\varepsilon \star A\VERT_{2\Omega}
	\leq \Vert \eta_\varepsilon \star A_n - \eta_\varepsilon \star A - \eta_\varepsilon \star\nabla\phi_n \Vert_{L^2(\Omega,\R^3)} \to 0 \,.
	$$
	This proves the proposition.
\end{proof}


\section{Nonzero weak limit}

Our goal in the section is to show that a minimizing sequence has a nonzero weak limit up to symmetries. In the language of concentration compactness, we exclude `vanishing'.


\subsection{An improved inequality}

Our goal in this section is to prove the following proposition, involving the seminorm $\VERT\cdot\VERT_3$ from \eqref{eq:defseminorm}.

\begin{proposition}\label{improved}
	For any $A\in\mathcal Y$,
	$$
	\VERT A \VERT_3 \lesssim \|\nabla\wedge A\|_{3/2}^{1/2} \left( \sup_{t>0} t\, \|e^{t\Delta} \nabla\wedge A\|_{L^{\infty}(\R^3,\R^3)} \right)^{1/2}.
	$$
\end{proposition}

Note that, since $e^{t\Delta}$ is convolution with a function whose $L^3$ norm is proportional to $t^{-1}$, one has, by H\"older,
$$
\sup_{t>0} t\, \|e^{t\Delta} \nabla\wedge A\|_{L^{\infty}(\R^3,\R^3)} \lesssim \|\nabla\wedge A\|_{3/2} \,.
$$
In this sense the inequality `improves' upon the Sobolev inequality in Lemma \ref{spositive}.

\begin{proof}
	We set $B:=\nabla\wedge A$ and define $\phi$ as in the proof of Lemma \ref{helmholtz}. Then $A-\nabla\phi=\tilde A$ is given by \eqref{eq:coulombgauge}. In this formula, we write
	$$
	\frac{1}{4\pi} \frac{x-y}{|x-y|^3} = -\frac{1}{4\pi} \nabla_x \frac{1}{|x-y|} = - \int_0^\infty dt\, (4\pi t)^{-3/2} \nabla_x e^{-(x-y)^2/4t} \,,
	$$
	where we used $(-\Delta)^{-1} = \int_0^\infty dt\, e^{t\Delta}$. Thus, with a parameter $T$ to be determined,
	$$
	A(x)-\nabla\phi(x) = I_<(x)+I_>(x) \,,
	$$
	where
	\begin{align*}
		I_<(x) & := - \int_0^T dt\, (4\pi t)^{-3/2} \int_{\R^3}dy\, B(y)\wedge \nabla_x e^{-(x-y)^2/4t} \,,\\
		I_>(x) & := - \int_T^\infty dt\, (4\pi t)^{-3/2} \int_{\R^3}dy\, B(y)\wedge \nabla_x e^{-(x-y)^2/4t} \,.
	\end{align*}
	
	Clearly,
	\begin{align*}
		|I_<(x)| & \leq \int_0^T dt\, (4\pi t)^{-3/2} \int_{\R^3}dy\, |B(y)| |\nabla_x e^{-(x-y)^2/4t}| \\
		& = \int_0^T dt\, t^{-2} \int_{\R^3}dy\, |B(y)| k((x-y)/\sqrt t)
	\end{align*}
	with
	$$
	k(z) := (4\pi)^{-3/2} (|z|/2) e^{-z^2/4} \,.
	$$
	Let $\tilde k$ be the monotone hull of $k$, that is,
	$$
	\tilde k(z) :=
	\begin{cases}
		(4\pi)^{-3/2} (1/\sqrt 2) e^{-1/2} & \text{if}\ |z|< \sqrt 2 \,,\\
		(4\pi)^{-3/2} (|z|/2) e^{-z^2/4} & \text{if}\ |z|\geq\sqrt 2 \,.
	\end{cases}
	$$
	Then, by the layer cake formula and the fact that the superlevel sets of $\tilde k$ are balls,
	\begin{align*}
		\int_{\R^3} t^{-3/2} k((x-y)/\sqrt t)|B(y)| \,dy
		& \leq \int_{\R^3} t^{-3/2} \tilde k((x-y)/\sqrt t)|B(y)| \,dy \\
		& = \int_0^\infty d\kappa\, t^{-3/2} \int_{\R^3} \mathbb 1_{\{\tilde k>\kappa\}}((x-y)/\sqrt{t}) |B(y)| \,dy \\
		& \leq \int_0^\infty d\kappa\, |\{\tilde k>\kappa\}| B^*(x) \\
		& = \|\tilde k\|_{L^1(\R^3)} B^*(x)
	\end{align*}
	with the maximal function
	$$
	B^*(x) := \sup_{r>0} \frac{3}{4\pi r^3}\int_{\{|x-y|<r\}}|B(y)|\,dy \,.
	$$
	Thus,
	$$
	|I_<(x)| \leq \|\tilde k\|_1 B^*(x) \int_0^T dt\,t^{-1/2} = 2 \|\tilde k\|_{L^1(\R^3)} B^*(x) T^{1/2} \,.
	$$
	
	On the other hand, to estimate $I_>$ we use the semigroup property $e^{t\Delta}=e^{t\Delta/2} e^{t\Delta/2}$ to write
	$$
	(4\pi t)^{-3/2} \int_{\R^3} dy\, B(y)\wedge\nabla_x e^{-(x-y)^2/4t}
	= (2\pi t)^{-3/2} \int_{\R^3} dy\, (e^{t\Delta/2} B)(y) \wedge \nabla_x e^{-(x-y)^2/2t} \,.
	$$
	Thus,
	\begin{align*}
		|I_>(x)| & \leq \int_T^\infty dt\, (2\pi t)^{-3/2} \int_{\R^3} dy\, |(e^{t\Delta/2}B)(y)| |\nabla_x e^{-(x-y)^2/2t}| \\
		& = \int_T^\infty dt \int_{\R^3} dy\, |(e^{t\Delta/2}B)(y)| (t/2)^{-2} k((x-y)/\sqrt{t/2}) \,.
	\end{align*}
	Now
	\begin{align*}
		\int_{\R^3} |(e^{t\Delta/2}B)(y)| (t/2)^{-2} k((x-y)/\sqrt{t/2}) \,dy
		& \leq \| e^{t\Delta/2}B\|_\infty  (t/2)^{-1/2} \|k\|_{L^1(\R^3)} \\
		& \leq M (t/2)^{-3/2} \|k\|_1
	\end{align*}
	with $M:=\sup_{t>0} t\|e^{t\Delta}B\|_\infty$. Thus,
	$$
	|I_>(x)|\leq \int_T^\infty dt\, M (t/2)^{-3/2} \|k\|_{L^1(\R^3)} = M \|k\|_{L^1(\R^3)} 4 (T/2)^{-1/2} \,.
	$$
	
	To summarize, we have shown that
	$$
	|A(x)-\nabla\phi(x)|\leq 2 \|\tilde k\|_{L^1(\R^3)} B^*(x) T^{1/2} + M \|k\|_{L^1(\R^3)} 4 (T/2)^{-1/2} \,.
	$$
	Optimizing in $T$, we get
	$$
	|A(x)-\nabla\phi(x)|\leq  2^{11/4} \|\tilde k\|_{L^1(\R^3)}^{1/2} \|k\|_{L^1(\R^3)}^{1/2} M^{1/2}  B^*(x)^{1/2} \,.
	$$
	We raise this inequality to the third power and use the fact that the maximal function is a bounded operator on $L^{3/2}(\R^3)$. This proves the claimed inequality.
\end{proof}

\begin{remark}
	The analogue of Proposition \ref{improved} for the Sobolev inequality for spinor fields mentioned in the introduction is
	\begin{equation}
		\label{eq:improvedspinor}
		\|\psi\|_3 \lesssim \|\sigma\cdot(-i\nabla)\psi\|_{3/2}^{1/2} \left( \sup_{t>0} t \| e^{t\Delta} \sigma\cdot(-i\nabla)\psi \|_\infty \right)^{1/2} \,,
		\qquad \psi\in\dot W^{1,3/2}(\R^3,\C^2) \,.
	\end{equation}
	This is proved in exactly the same way as Proposition \ref{improved}.
\end{remark}



\subsection{Nonzero weak limit up to symmetries}

In this subsection we use the improved inequality in Proposition \ref{improved} to show that, up to translations and dilations, one can extract from every minimizing sequence for $S$ a subsequence which has a nontrivial weak limit.

\begin{proposition}\label{nonzero}
	Let $(A_n)\subset \mathcal Y$ be a sequence with
	$$
	\VERT A_n \VERT_3 =1
	\qquad\text{and}\qquad
	\| \nabla\wedge A_n \|_{3/2} \lesssim 1 \,.
	$$
	Then there are $\lambda_n\in(0,\infty)$ and $a_n\in\R^3$ such that a subsequence of
	$$
	\lambda_n^2 (\nabla\wedge A_n)(\lambda_n(x-a_n))
	$$
	converges weakly in $L^{3/2}(\R^3,\R^3)$ to some $\tilde B\not\equiv 0$.
\end{proposition}

\begin{proof}
	We write $B_n:=\nabla\wedge A_n$. Applying the improved inequality from Proposition \ref{improved}, we infer that
	$$
	\epsilon := \liminf_{n\to\infty} \sup_{t>0} t\, \|e^{t\Delta} B_n\|_{L^\infty(\R^3,\R^3)} >0 \,.
	$$
	Thus, for all sufficiently large $n$ there are $t_n>0$ and $x_n\in\R^3$ such that
	$$
	t_n | (e^{t_n\Delta} B_n)(x_n)|\geq \epsilon/2 \,,
	$$
	that is,
	\begin{equation}
		\label{eq:weaklimitnonzero}
		\left| \int_{\R^3} (4\pi)^{-3/2} e^{-x^2/4} \tilde B_n(x)\,dx \right|\geq \epsilon/2
	\end{equation}
	for
	$$
	\tilde B_n(x) := t_n B_n(\sqrt{t_n} x + x_n) \,.
	$$
	Since $\|\tilde B_n\|_{3/2} = \|B_n\|_{3/2}=1$, weak compactness implies that a subsequence of $(\tilde B_n)$ converges weakly in $L^{3/2}(\R^3,\R^3)$ to some $\tilde B$. Since $e^{-x^2/4}\in L^3(\R^3)$, it follows from \eqref{eq:weaklimitnonzero} that $\tilde B\not\equiv 0$, as claimed.
\end{proof}


\section{Applying Ekeland's variational principle}

In the previous section, in order to get a nonzero weak limit along a subsequence, we only used $\limsup_{n\to\infty} \left( \|\nabla\wedge A_n \|_{3/2}/\VERT A_n \VERT_3 \right) >0$. We did \emph{not} use the fact that $A_n$ is minimizing for \eqref{eq:defs}. Our goal in this and the next section is to upgrade the weak convergence to strong convergence, and we do this by using $\lim_{n\to\infty} \left( \|\nabla\wedge A_n \|_{3/2}/\VERT A_n\VERT_3 \right)$ $= S^{2/3}$ with $S$ from \eqref{eq:defs}.

More specifically, in this section we show that the minimizing sequence can be slightly altered to satisfy a version of the Euler--Lagrange equation with a small inhomogeneity. This will be achieved through Ekeland's variational principle. In the next section, we will study this approximated Euler--Lagrange equation in more detail.

\begin{proposition}\label{ekelandappl}
	Let $(A_n)\subset\mathcal Y$ be a minimizing sequence for $S$ with $\VERT A_n \VERT_3 = 1$. Then there is a sequence $(A_n')\subset\mathcal Y$ with $\nabla\cdot(|A_n'|A_n')=0$ and $\|A_n'\|_3 =1$ for all $n$ such that	
	$$
	\nabla\wedge A_n'- \nabla\wedge A_n \to 0
	\qquad\text{in}\ L^{3/2}(\R^3,\R^3)
	$$
	and
	$$
	\nabla\wedge  (|\nabla\wedge A_n'|^{-1/2} \nabla\wedge A_n') - S |A_n'| A_n'  = \nabla\wedge r_n
	\qquad\text{with}\ r_n \to 0 \ \text{in}\ L^3(\R^3,\R^3) \,.
	$$
\end{proposition}

Our goal in this section will be to prove this proposition. 


\subsection{Differentiability of the seminorm}

The following result implies that $A\mapsto \VERT A \VERT_3^3$ is Fr\'echet differentiable and gives a formula for its derivative.

\begin{lemma}\label{differentiability}
	Let $A\in L^3(\R^3,\R^3)$ with $\nabla\cdot (|A|A) = 0$. Then for all $F\in L^3(\R^3,\R^3)$,
	$$
	\left| \VERT A + F \VERT_3^3 - \VERT A \VERT_3^3 - 3 \int_{\R^3} |A|A\cdot F \,dx \right| \lesssim \VERT A \VERT_3^{3/2} \VERT F \VERT_3^{3/2} + \VERT F \VERT_3^3 \,.
	$$
\end{lemma}

\begin{proof}
	Note that according to Lemma \ref{gaugeinf}, the assumption on $A$ implies $\VERT A \VERT_{3} = \|A\|_3$. Moreover, since the claimed inequality is invariant under adding a gradient to $F$, we may also assume that $\VERT F \VERT_{3} = \|F\|_3$. We have
	$$
	\VERT A+F \VERT_{3}^3 \leq \|A+F\|_3^3 \leq \|A\|_3^3 + 3 \int_{\R^3} |A|A\cdot F \,dx + \const \!\! \left( \|A\|_3 \|F\|_3^2 + \|F\|_3^3 \right).
	$$
	Since the remainder on the right side is bounded by a constant times $\|A\|_3^{3/2} \|F\|_3^{3/2} + \|F\|_3^3$, this proves one of the two claimed inequalities. For the converse inequality, we choose $\phi$ such that $\VERT A+F \VERT_3 = \|A+F-\nabla\phi\|_3$ and bound
	\begin{align*}
		\VERT A+F \VERT_{3}^3 & \geq \|A-\nabla\phi\|_3^3 + 3 \int_{\R^3} |A-\nabla\phi|(A-\nabla\phi)\cdot F \,dx \\
		& \quad - \const \!\! \left( \|A-\nabla\phi\|_3 \|F\|_3^2 + \|F\|_3^3 \right).
	\end{align*}
	To bound the right side, we use $\|A-\nabla\phi\|_3^3 \geq \VERT A \VERT_{3}^3 = \|A\|_3^3$ and
	$$
	\int_{\R^3} |A-\nabla\phi|(A-\nabla\phi)\cdot F \,dx \geq \int_{\R^3} |A|A\cdot F \,dx
	- 2 \|\nabla\phi\|_3 \|A\|_3 \|F\|_3 - \|\nabla\phi\|_3^2 \|F\|_3 \,.
	$$
	In this way we arrive at
	\begin{align*}
		\VERT A+F \VERT_{3}^3 & \geq \VERT A \VERT_{3}^3 + 3\int_{\R^3} |A|A\cdot F \,dx \\
		& \quad - \const\!\! \left( \|A\|_3 \|F\|_3^2 + \|F\|_3^3 + \|\nabla\phi\|_3 \|A\|_3 \|F\|_3 + \|\nabla\phi\|_3^2 \|F\|_3 \right)
	\end{align*}
	and it remains to bound $\|\nabla\phi\|_3$. We note that
	$$
	\nabla\cdot(|A+F-\nabla\phi|(A+F-\nabla\phi)) = 0 
	\qquad\text{and}\qquad
	\nabla\cdot(|A|A) = 0 \,.
	$$
	Using the $\R^3$-version of Lemma \ref{contbound} corresponding to $\epsilon=0$ (which can be proved by the same argument and is, in fact, much simpler since for $\epsilon=0$ Theorem \ref{nonlinearhelm} is not needed) we obtain
	$$
	\| \nabla \phi \|_3^2 \lesssim \|F\|_3 \left( \|A+F\|_3 + \|A\|_3 \right).
	$$
	Inserting this into the above bound, we obtain the lower bound
	$$
	\VERT A+F \VERT_{3}^3 \geq \VERT A \VERT_{3}^3 + 3\int_{\R^3} |A|A\cdot F \,dx - \const\!\! \left( \|A\|_3^{3/2} \|F\|_3^{3/2} + \|F\|_3^3 \right),
	$$
	which concludes the proof.
\end{proof}


\subsection{Applying Ekland's variational principle}\label{sec:ekelandappl}

In order to prove Proposition \ref{ekelandappl} we apply Ekeland's variational principle in the following setting. We recall that $\mathcal Y$ and $\mathcal M$ were defined in \eqref{eq:defy} and \eqref{eq:defm}, respectively, and we set
$$
X := \mathcal Y /\mathcal M
$$
endowed with the norm $\|\nabla\wedge A\|_{3/2}$. Using standard properties of weak derivatives it is easy to see that this is, indeed, a norm and that $X$ endowed with this norm is complete. We claim that the dual space of $X$ is
\begin{equation}
	\label{eq:dualalt}
	X^* = \left\{ \nabla\wedge B :\ B \in L^3(\R^3,\R^3) \right\}
\end{equation}
with norm
$$
\| \nabla\wedge B\|_{X^*} = \inf_{\phi\in\dot W^{1,3}(\R^3)} \|B-\nabla\phi\|_3 \,,
$$
in the sense that every functional $\Phi\in X^*$ is of the form
$$
\Phi([A]) = \int_{\R^3} B \cdot(\nabla\wedge A)\,dx  
\qquad\text{for}\ [A] \in X \,,
$$
and conversely, any functional of this form defines an element of $X^*$. This can again be shown by standard arguments.

From Ekeland's variational principle we will deduce the following lemma, which is the core of the proof of Proposition \ref{ekelandappl}.

\begin{lemma}\label{ekelandappllem}
	Let $A\in\mathcal Y$ with $\VERT A \VERT_3 = 1$. Then for any $\delta>0$ there are $A'\in\mathcal Y$ and $\lambda\in\R$ such that
	$$
	\nabla\cdot(|A'|A')=0 \,,\qquad \|A'\|_3 = \VERT A' \VERT_3 = 1 \,,
	$$
	$$
	\|\nabla\wedge A'\|_{3/2} \leq \|\nabla\wedge A\|_{3/2} \,,
	\qquad
	\|\nabla\wedge (A'-A)\|_{3/2} \leq\delta \,,
	$$
	and
	$$
	\left\| \nabla\wedge ( |\nabla\wedge A'|^{-1/2} \nabla\wedge A') - \lambda |A'|A' \right\|_{X^*} \leq 3 \left( \|\nabla\wedge A\|_{3/2}^{3/2} - S \right) \delta^{-1} \,.
	$$
\end{lemma}

\begin{proof}[Proof of Proposition \ref{ekelandappl} given Lemma \ref{ekelandappllem}]
	We apply Lem\-ma \ref{ekelandappllem} to $A=A_n$ with the choice $\delta^2= \epsilon_n:= \|\nabla\wedge A_n\|_{3/2}^{3/2} - S$. We obtain sequences $(A_n')\subset\mathcal Y$ and $(\lambda_n)\subset\R$ such that
	$$
	\nabla\cdot(|A_n'|A_n')=0 \,,
	\qquad \|A_n'\|_3 = 1 \,,
	$$
	\begin{equation}
		\label{eq:ekelandappl0}
		\| \nabla\wedge (A_n'-A_n) \|_{3/2} \leq \sqrt{\epsilon_n}
	\end{equation}
	and
	\begin{equation}
		\label{eq:ekelandappl2}
		\left\| \nabla\wedge ( |\nabla\wedge A_n'|^{-1/2} \nabla\wedge A_n') - \lambda_n |A_n'|A_n' \right\|_{X^*} \leq 3 \sqrt{\epsilon_n} \,.
	\end{equation}
	Since $(A_n)$ is minimizing, we have $\epsilon_n\to 0$ and therefore, by \eqref{eq:ekelandappl0}, $\nabla\wedge (A_n'-A_n)\to 0$ in $L^{3/2}$, as claimed.
	
	Let us show that in the almost-Euler--Lagrange equation \eqref{eq:ekelandappl2}, we can replace $\lambda_n$ by $S$. Indeed, since $([A_n'])$ is bounded in $X$ by \eqref{eq:ekelandappl0}, when testing \eqref{eq:ekelandappl2} against $[A_n']$ we obtain
	\begin{equation}
		\label{eq:ekelandappl3}
		\int_{\R^3} |\nabla\wedge A_n'|^{3/2}\,dx - \lambda_n \to 0 \,.
	\end{equation}
	Here we also used $\|A_n'\|_3 = 1$. According to \eqref{eq:ekelandappl0} and the minimizing property of $A_n$ we have $\int_{\R^3} |\nabla\wedge A_n'|^{3/2}\,dx \to S$. Therefore \eqref{eq:ekelandappl3} implies $\lambda_n\to S$ and
	\begin{align*}
		& \left\| \nabla\wedge (|\nabla\wedge A_n'|^{-1/2} \nabla\wedge A_n') - S |A_n'| A_n' \right\|_{X^*} \\
		& \qquad \leq \left\| \nabla\wedge(|\nabla\wedge A_n'|^{-1/2}\nabla\wedge A_n') - \lambda_n |A_n'| A_n'  \right\|_{X^*} + |\lambda_n - S| \| |A_n'| A_n' \|_{X^*} \to 0 \,.
	\end{align*}
	Here we used the fact that $|A_n'| A_n'$ is uniformly bounded in $X^*$. Indeed, the Sobolev inequality in Lemma \ref{spositive} implies by duality that $(L^3/\mathcal M)^* \subset X^*$ continuously and so,
	$$
	\| |A_n'| A_n' \|_{X^*}\lesssim \| |A_n'| A_n' \|_{(L^3/\mathcal M)^*} \leq \| |A_n'|A_n\|_{3/2} = \|A_n'\|_{3}^2 = 1 \,.
	$$
	
	To complete the proof, we recall the characterization of $X^*$ in \eqref{eq:dualalt}, which implies that there is an $r_n\in L^3(\R^3,\R^3)$ such that
	$$
	\nabla\wedge (|\nabla\wedge A_n'|^{-1/2} \nabla\wedge A_n') - S |A_n'| A_n'  = \nabla\wedge r_n \,.
	$$
	After subtracting a gradient, we can assume that $\nabla\cdot(|r_n|r_n)=0$ and then
	$$
	\| r_n\|_3 = \left\| \nabla\wedge (|\nabla\wedge A_n'|^{-1/2} \nabla\wedge A_n') - S |A_n'| A_n' \right\|_{X^*} \to 0 \,.
	$$
	This completes the proof of the proposition.
\end{proof}

It remains to prove Lemma \ref{ekelandappllem}. This is almost an immediate consequence of \cite[Theorem 3.1]{E1}, except that it is not obvious to us that the functional $A\mapsto \VERT A \VERT_{3}$ is \emph{continuously} Fr\'echet differentiable. Its Fr\'echet differentiability is a consequence of Lemma \ref{differentiability}. While it might be possible to show the continuity of its Fr\'echet derivative, we think it is easier to redo in our setting the reduction of \cite[Theorem 3.1]{E1} to \cite[Theorem 1.1]{E1}. The observation is that because of the homogeneity of $A\mapsto \VERT A \VERT_{3}$, one does not need its continuous Fr\'echet differentiability. In fact, only its Gateaux differentiability suffices. Our proof also shows that the same method works in the case of \emph{complex} Banach spaces. This substantiates our claim in the introduction that the same method works for the Sobolev inequality for spinor fields.

\begin{proof}[Proof of Lemma \ref{ekelandappllem}]
	We consider the metric space $Z:=\{[A]\in X:\ \VERT A \VERT_{3}=1\}$ with the metric induced by the norm in $X$. As a consequence of the Sobolev inequality in Lemma \ref{spositive}, $Z$ is a closed subset of $X$ and therefore complete. In $Z$ we consider the functional $F([A]) := \|\nabla\wedge A\|_{3/2}$, which is well-defined and continuous. Now given $A\in\mathcal Y$ with $\VERT A \VERT_{3}=1$ and $\delta>0$, we deduce from \cite[Theorem 1.1]{E1} that there is an $[A']\in Z$ such that
	$$
	\|\nabla\wedge A'\|_{3/2} \leq \|\nabla\wedge A\|_{3/2} \,,
	\qquad
	\|\nabla\wedge (A'-A)\|_{3/2} \leq\delta \,,
	$$
	and such that for any $[A']\neq[A'']\in Z$,
	\begin{equation}
		\label{eq:ekelandappllemproof}
		\|\nabla\wedge A''\|_{3/2}^{3/2} > \|\nabla\wedge A'\|_{3/2}^{3/2} - \frac{\epsilon}{\delta}\ \|\nabla\wedge (A''-A')\|_{3/2}
	\end{equation}
	with $\epsilon:= \|\nabla\wedge A\|_{3/2}^{3/2} - S$. According to Lemma \ref{gaugeinf} we can fix the gauge of $A'$ such that $\nabla\cdot(|A'|A')=0$ and then $\|A'\|_3=1$.
	
	It remains to prove the last inequality in the statement of the lemma. Let $F\in\mathcal Y$ with
	\begin{equation}
		\label{eq:ekelandappllemproof2}
		\int_{\R^3} |A'|A'\cdot F\,dx=0 \,.
	\end{equation}
	We apply \eqref{eq:ekelandappllemproof} with $A'' = (A'+tF)/\VERT A'+tF \VERT_{3}$, where $t\geq 0$ is such that $\VERT A'+tF \VERT_{3}\neq 0$. We conclude that
	$$
	\frac{\|\nabla\!\wedge\!(A'+tF)\|_{3/2}^{3/2}}{ \VERT A'+tF \VERT_{3}^{3/2}} \geq \|\nabla\wedge A'\|_{3/2}^{3/2} - \frac{\epsilon}{\delta}\ \frac{t}{\VERT A'+tF \VERT_{3}} 
	\left\|\nabla\! \wedge \!\left( \!F - \frac{\VERT A'+tF \VERT_{3}-1}{t} A' \right) \right\|_{3/2}.
	$$
	Now we use the Gateau differentiability of $\|\cdot\|_{3/2}$ and $\VERT \cdot \VERT_{3}$. Note that Lemma \ref{differentiability} and \eqref{eq:ekelandappllemproof2} imply that
	$$
	\VERT A'+tF \VERT_{3} = 1+ o(t) \,.
	$$
	Thus, we obtain
	\begin{equation}
		\label{eq:ekelandappllemproof3}
		\frac32 \int_{\R^3} |\nabla\wedge A'|^{-1/2} (\nabla\wedge A')\cdot(\nabla\wedge F)\,dx
		\geq - \frac{\epsilon}{\delta} \|\nabla \wedge F \|_{3/2} \,.
	\end{equation}
	By a simple abstract result (see Lemma \ref{lagrange} below), from the fact that for any $F\in\mathcal Y$, \eqref{eq:ekelandappllemproof2} implies \eqref{eq:ekelandappllemproof3}, we deduce the existence of a $\lambda\in\R$ such that the last inequality in the statement of Lemma \ref{ekelandappllem} holds.
\end{proof}

The following lemma holds for normed spaces over $\mathbb K$, where either $\mathbb K=\R$ or $\mathbb K=\C$. In the real case it is a special case of \cite[Lemma 3.3]{E1}. If one were to apply our techniques to the case of spinor fields, one would need the complex case.

\begin{lemma}\label{lagrange}
	Let $X$ be a normed space and let $F,G\in X^*$ and $\rho>0$ such that for any $x\in X$ with $\langle G,x\rangle=0$ one has $\re\langle F, x\rangle \geq -\rho\|x\|$. Then there is a $\lambda\in\mathbb K$ such that
	$$
	\|F - \lambda G \|\leq\rho \,.
	$$
\end{lemma}

\begin{proof}
	By applying the assumption to $x$ times a constant of absolute value one, we see that $\left| \langle F,x\rangle \right|\leq \rho \|x\|$ for all $x\in X$ with $\langle G,x\rangle = 0$. Thus, $ \| F|_{\ker G} \|\leq\rho$. By Hahn--Banach, there is an $\tilde F\in X^*$ such that $\tilde F|_{\ker G} = F|_{\ker G}$ and $\|\tilde F\| = \| F|_{\ker G} \|$. In particular, $\|\tilde F\| \leq\rho$. By construction, $\ker G \subset \ker (\tilde F - F)$. Thus, by a well-known algebraic lemma \cite[Lemma 3.2]{Br}, there is a $\lambda\in\mathbb K$ such that $\tilde F- F = \lambda G$. Thus, $\|F- \lambda G\| = \|\tilde F\|\leq\rho$, as claimed.
\end{proof}


\section{Study of the approximate Euler--Lagrange equation}

In this section we study solutions $A_n'$ to the equations
\begin{equation}
	\label{eq:eq}
	\nabla\wedge ( |\nabla\wedge A_n'|^{-1/2}(\nabla\wedge A_n')) - S |A_n'|A_n' 
	= \nabla \wedge r_n
	\qquad\text{with}\ r_n \to 0 \ \text{in}\ L^{3}(\R^3,\R^3)
\end{equation}
satisfying
\begin{equation}
	\label{eq:studyeq}
	\nabla \wedge A_n' \rightharpoonup B
	\qquad\text{in}\ L^{3/2}(\R^3,\R^3) \,.
\end{equation}
The constant $S$ is defined in \eqref{eq:defs}. The functions $A_n'$ are not necessarily those constructed in Proposition \ref{ekelandappl}, although this is the application that we have in mind. Note that \eqref{eq:eq} and \eqref{eq:studyeq} imply
\begin{equation}
	\label{eq:studyeq1}
	\| A_n' \|_3 \lesssim 1 \,.
\end{equation}
Indeed, \eqref{eq:eq} implies $\nabla\cdot(|A_n'|A_n')=0$ and therefore, by Lemma \ref{gaugeinf}, $\|A_n'\|_3 = \VERT A_n' \VERT_3$. The boundedness of $(\nabla\wedge A_n')$ in $L^{3/2}$, which follows from \eqref{eq:studyeq}, and Lemma \ref{spositive} give \eqref{eq:studyeq1}.

Our goal in the two subsections of this section is to prove two lemmas concerning the derivative and the nonderivative terms, respectively, on the left side of \eqref{eq:eq}.


\subsection{The truncation argument}

We will prove the following convergence result.

\begin{lemma}\label{goal22}
	In the situation \eqref{eq:eq}, \eqref{eq:studyeq}, we have $\nabla\wedge A_n' \to B$ in $L^p_\loc(\R^3,\R^3)$ for any $p<3/2$.
\end{lemma}

This lemma is in the spirit of convergence theorems for quasilinear equations due to Boccardo and Murat \cite{BoMu} and, in the case of systems, Dal Maso and Murat \cite{DaMu}. While our equation does not satisfy the assumptions in \cite{DaMu}, after choosing an appropriate gauge we can follow their argument rather closely.

\begin{proof}
	We prove that every subsequence has a further subsequence along which we have convergence in $L^p_\loc(\R^3,\R^3)$ for any $p<3/2$. This clearly implies the lemma.
	
	We abbreviate $B_n:=\nabla\wedge A_n'$. We apply Lemma \ref{helmholtz} to $A_n'$ and obtain $\tilde A_n$ such that $\nabla\wedge\tilde A_n =B_n$ and $\nabla\cdot\tilde A_n = 0$. Moreover, the bound in Lemma \ref{helmholtz} together with the $L^{3/2}$-boundedness of $(B_n)$ resulting from \eqref{eq:studyeq} implies that $\tilde A_n$ is bounded in $\dot W^{1,3/2}(\R^3,\R^3)$. Thus, after passing to a subsequence, we have $\tilde A_n\rightharpoonup \tilde A$ in $\dot W^{1,3/2}(\R^3,\R^3)$. The first consequence of this convergence is that $B_n = \nabla\wedge\tilde A_n \rightharpoonup \nabla\wedge\tilde A$ in $L^{3/2}(\R^3,\R^3)$ and therefore $\nabla\wedge\tilde A=B$. The second consequence is that by the Rellich--Kondrachov lemma, $\tilde A_n\to\tilde A$ in $L^q_\loc(\R^3,\R^3)$ for any $q<3$. Moreover, after passing to a further subsequence, $\tilde A_n \to\tilde A$ almost everywhere.
	
	Let $\psi\in C^1(\R^3,\R^3)$ with $\psi(y)=y$ for $|y|\leq 1$ and $\psi(y)=0$ for $|y|\geq 2$. For $\delta>0$ we set $\psi_{\delta}(y):=\delta \psi(y/\delta)$. Let $\chi\in C^1_c(\R^3)$ and let $(\delta_n)\subset(0,\infty)$ be a bounded sequence to be specified later and multiply equation \eqref{eq:eq} by $\chi \psi_{\delta_n}(\tilde A_n-\tilde A)$ to obtain
	\begin{align*}
		& \int_{\R^3} \chi \left( |B_n|^{-1/2}B_n - |B|^{-1/2}B \right) \cdot \left( \nabla\wedge \psi_{\delta_n}(\tilde A_n-\tilde A)\right)dx \\
		& = - \int_{\R^3} |B_n|^{-1/2}B_n \cdot \!\left( \nabla\chi \wedge \psi_{\delta_n}(\tilde A_n-\tilde A) \right)\!dx  
		- \int_{\R^3} \chi |B|^{-1/2}B \cdot \!\left(\nabla\wedge \psi_{\delta_n}(\tilde A_n-\tilde A)\right)\!dx \\
		& \quad + S \int_{\R^3} \chi |A_n'| A_n' \cdot \psi_{\delta_n}(\tilde A_n-\tilde A)\,dx
		+ \int_{\R^3} r_n \cdot \nabla\wedge \left( \chi \psi_{\delta_n}(\tilde A_n-\tilde A)\right) dx \,.
	\end{align*}
	It is not difficult to see that, independently of the choice of $(\delta_n)$, $\nabla\wedge\psi_{\delta_n}(\tilde A_n-\tilde A)\rightharpoonup 0$ in $L^{3/2}(\R^3,\R^3)$. This implies that the second term on the right side tends to zero as $n\to\infty$. Moreover, since $\nabla\wedge\psi_{\delta_n}(\tilde A_n-\tilde A)$ is bounded in $L^{3/2}(\R^3)$ and $r_n$ tends to zero in $L^3(\R^3,\R^3)$, the fourth term on the right side tends to zero as $n\to\infty$. Thus,
	\begin{align}\label{eq:bmproof1}
		& \limsup_{n\to\infty} \int_{\R^3} \chi \left( |B_n|^{-1/2}B_n - |B|^{-1/2}B \right) \cdot \left(\nabla\wedge \psi_{\delta_n}(\tilde A_n-\tilde A)\right) dx \notag \\
		& \leq \limsup_{n\to\infty} \left( \| B_n\|_{3/2}^{1/2} \|\nabla\chi\|_{3/2} \|\psi_{\delta_n}(\tilde A_n-\tilde A)\|_\infty + S \|A_n'\|_3^2 \|\chi\|_{3} \|\psi_{\delta_n}(\tilde A_n-\tilde A)\|_\infty \right) \notag \\
		& \leq \left( \Gamma \|\nabla\chi\|_{3/2} + S \Gamma' \|\chi\|_3 \right)  M \limsup_{n\to\infty} \delta_n
	\end{align}
	with $M:= \sup |\psi|$, $\Gamma:=\limsup_{n\to\infty} \| B_n \|_{3/2}^{1/2}$ and $\Gamma':=\limsup_{n\to\infty} \|A_n'\|_3^2$. Note that $\Gamma$ and $\Gamma'$ are finite by \eqref{eq:studyeq} and \eqref{eq:studyeq1}.	We bound
	\begin{align}\label{eq:bmproof2}
		& \int_{\R^3} \chi \left( |B_n|^{-1/2}B_n - |B|^{-1/2}B \right) \cdot \left( \nabla\wedge \psi_{\delta_n}(\tilde A_n-\tilde A)\right) dx \notag \\
		& \geq \int_{\{ |\tilde A_n-\tilde A|\leq\delta_n\}} \chi e_n\,dx - \|\chi\|_\infty M' \int_{\{ \delta_n<|\tilde A_n-\tilde A|\leq 2\delta_n\}} h_n\,dx
	\end{align}
	with $M':= \sup |\nabla\otimes \psi|$ and
	\begin{align*}
		e_n & := \left( |B_n|^{-1/2}B_n - |B|^{-1/2}B \right) \cdot (B_n-B) \,, \\
		h_n & := \left( |B_n|^{1/2} + |B|^{1/2} \right) |\nabla\otimes (\tilde A_n-\tilde A)| \,.
	\end{align*}
	Here we used $|\nabla\wedge F|\leq |\nabla\otimes F|$ and $|\nabla\otimes \psi(G)|\leq M' |\nabla\otimes G|$.
	
	By Lemma \ref{helmholtz}, with $C$ denoting the implicit constant in the first bound there,
	\begin{align*}
		\int_{\R^3} h_n\,dx & \leq \left\| |B_n|^{1/2} + |B|^{1/2} \right\|_3 \left\| \nabla \otimes (\tilde A_n-\tilde A) \right\|_{3/2} \\
		& \leq C \left( \|B_n\|_{3/2}^{1/2} + \|B\|_{3/2}^{1/2} \right)\left( \|B_n\|_{3/2} + \|B\|_{3/2} \right),
	\end{align*}
	so $\limsup_{n\to\infty} \|h_n\|_1 \leq 4C \Gamma^{3}$ and, in particular, $h_n$ is bounded in $L^1$.
	
	We fix two parameters $0<\epsilon<\epsilon'$ and choose the sequence $(\delta_n)$ depending on those parameters as follows. We have
	$$
	\int_{\epsilon}^{\epsilon'} \int_{\{\delta<|\tilde A_n-\tilde A|\leq 2\delta\}} h_n\,dx \,\frac{d\delta}{\delta} \leq \int_{\R^3} h_n \int_{|\tilde A_n(x)-\tilde A(x)|/2}^{|\tilde A_n(x)-\tilde A(x)|} \frac{d\delta}{\delta} \,dx = (\ln 2) \, \|h_n\|_1
	$$
	and
	$$
	\int_\epsilon^{\epsilon'} \frac{d\delta}{\delta} = \ln\frac{\epsilon'}{\epsilon} \,.
	$$
	Thus, for each $n$ there is a $\delta_n\in[\epsilon,\epsilon']$ such that
	$$
	\int_{\{\delta_n<|\tilde A_n-\tilde A|\leq 2\delta_n\}} h_n\,dx \leq \frac{(\ln 2) \, \|h_n\|_1}{\ln(\epsilon'/\epsilon)} \,.
	$$
	From now on, we work with this choice of $\delta_n$.
	
	It is elementary to see that
	\begin{equation}
		\label{eq:elementary}
		\left( |v|^{-1/2} v - |w|^{-1/2}w \right) \cdot(v-w) \geq (|v|^2+|w|^2)^{-1/4}|v-w|^2
		\qquad\text{for all}\ v,w\in\R^3 \,,
	\end{equation}
	and therefore
	\begin{equation}
		\label{eq:bmproof3}
		e_n \geq \left( |B_n|^2 + |B|^2\right)^{-1/4}|B_n-B|^2 \geq 0 \,.
	\end{equation}
	Assuming, in addition, that $\chi\geq 0$, we can bound, using the choice of $\delta_n$ and \eqref{eq:bmproof2},
	\begin{align*}
		\int_{\{|\tilde A_n-\tilde A|\leq\epsilon\}} \chi e_n\,dx & \leq \int_{\{|\tilde A_n-\tilde A|\leq\delta_n\}} \chi e_n\,dx \\
		& \leq \int_{\R^3} \chi \left( |B_n|^{-1/2}B_n - |B|^{-1/2}B \right) \cdot \left(\nabla\wedge \psi_{\delta_n}(\tilde A_n-\tilde A)\right)dx \\
		& \quad + \frac{(\ln2) \|\chi\|_\infty M' \|h_n\|_1}{\ln(\epsilon'/\epsilon)} \,.
	\end{align*}
	Thus, in view of \eqref{eq:bmproof1},
	\begin{align*}
		\limsup_{n\to\infty} \int_{\{|\tilde A_n-\tilde A|\leq\epsilon\}} \chi e_n\,dx & \leq \left( \Gamma \|\nabla\chi\|_{3/2} + S \Gamma' \|\chi\|_3 \right)  M \limsup_{n\to\infty} \delta_n \\
		& \qquad + \frac{(\ln2) \|\chi\|_\infty M' \limsup_{n\to\infty} \|h_n\|_1}{\ln(\epsilon'/\epsilon)} \\
		& \leq \left( \Gamma \|\nabla\chi\|_{3/2} + S \Gamma' \|\chi\|_3 \right)  M \epsilon' + \frac{4C \Gamma^3 (\ln2) \|\chi\|_\infty M'}{\ln(\epsilon'/\epsilon)} \,.
	\end{align*}
	
	Now let $0<\theta<1$ and bound
	$$
	\int_{\R^3} \chi e_n^\theta \,dx
	\leq \|\chi\|_1^{1-\theta} \left( \int_{\{|\tilde A_n-\tilde A|\leq\epsilon\}}\!\! \chi e_n\,dx \right)^{\theta} + \|\chi\|_\infty^\theta  \left( \int_{\R^3} e_n\,dx \right)^\theta \! \left( \int_{\{|\tilde A_n-\tilde A|>\epsilon\}} \!\! \chi \,dx \right)^{1-\theta}.
	$$
	Since $e_n$ is bounded in $L^1(\R^3)$ and since $\tilde A_n\to \tilde A$ almost everywhere, dominated convergence implies that
	$$
	\limsup_{n\to\infty} \int_{\R^3} \chi e_n^\theta \,dx \leq \|\chi\|_1^{1-\theta} \limsup_{n\to\infty} \left( \int_{\{|\tilde A_n-\tilde A|\leq\epsilon\}} \chi e_n\,dx \right)^{\theta}. 
	$$ 
	Inserting the bound on the limsup on the right side, we obtain
	$$
	\limsup_{n\to\infty} \int_{\R^3} \chi e_n^\theta \,dx \leq \|\chi\|_1^{1-\theta} 
	\left( \left( \Gamma \|\nabla\chi\|_{3/2} + S \Gamma' \|\chi\|_3 \right)  M \epsilon' + \frac{4C \Gamma^3 (\ln2) \|\chi\|_\infty M'}{\ln(\epsilon'/\epsilon)} \right)^{\theta}. 
	$$
	Letting first $\epsilon\to 0$ and then $\epsilon'\to 0$, we find
	$$
	\lim_{n\to\infty} \int_{\R^3} \chi e_n^\theta \,dx = 0 \,,
	$$
	and therefore $e_n^\theta\to 0$ in $L_\loc^1(\R^3)$. According to the following lemma, this implies $B_n\to B$ in $L^p_\loc(\R^3,\R^3)$ for any $p<3/2$. This completes the proof.
\end{proof}

\begin{lemma}
	Let $E$ be a set of finite measure and let $(F_n)\subset L^{3/2}(E,\R^3)$ be bounded. Assume that for some $F\in L^{3/2}(E,\R^3)$ and some $\theta>0$, one has
	$$
	\left( \left(|F_n|^{-1/2} F_n - |F|^{-1/2} F\right)\cdot (F_n-F)\right)^\theta \to 0
	\qquad\text{in}\ L^1(E) \,.
	$$
	Then $F_n\to F$ in $L^p(E)$ for any $1\leq p<3/2$.
\end{lemma}

\begin{proof}
	We show that any subsequence has a further subsequence along which $F_n\to F$ in $L^p(E)$ for any $p<3/2$. This clearly implies the lemma.
	
	The assumption implies that, given a subsequence, there is a further subsequence along which one has $\left(|F_n|^{-1/2} F_n - |F|^{-1/2} F\right)\cdot (F_n-F)\to 0$ almost everywhere in $E$. By \eqref{eq:elementary}, this implies $F_n\to F$ almost everywhere on $E$. By Vitali's convergence theorem, this, together with the boundedness of $(F_n)$ in $L^{3/2}$, implies the assertion.
\end{proof}

Looking back at the proof of Lemma \ref{goal22}, one might wonder why we passed from $A_n'$ to $\tilde A_n$. This was needed at two places. First, in the bound on $\int \chi e_n^\theta\,dx$ we used the fact that $\tilde A_n\to\tilde A$ almost everywhere and second, in the bound on $\nabla\wedge\psi_{\delta_n}(\tilde A_n-\tilde A)$, we needed the full Jacobi matrix $\nabla \otimes (\tilde A_n-\tilde A)$ and not only $\nabla\wedge(\tilde A_n-\tilde A)$. While we could get around the first item by appealing to Theorem \ref{strongconv}, we do not know how perform the truncation in the second item in the gauge of Lemma \ref{gaugeinf}.



\subsection{Application of the nonlinear Rellich--Kondrachov lemma}

In the previous subsection we have proved a convergence result about $\nabla\wedge A_n'$. Independently of that, we will now prove a convergence result for $A_n'$. This is where the nonlinear Rellich--Kondrachov lemma in Theorem \ref{strongconv} enters.

\begin{lemma}\label{goal21}
	In the situation \eqref{eq:eq}, \eqref{eq:studyeq}, there is an $A'\in\mathcal Y$ with $\nabla\wedge A'=B$ and $\nabla\cdot(|A'|A')=0$ such that $A_n'\rightharpoonup A'$ in $L^3(\R^3,\R^3)$ and $A_n' \to A'$ in $L^q_\loc(\R^3,\R^3)$ for any $q<3$.
\end{lemma}

\begin{proof}
	Let $A'\in\mathcal Y$ with $\nabla\wedge A'=B$ and $\nabla\cdot(|A'|A')=0$. Such an $A'$ exists, for define $\tilde A$ by \eqref{eq:coulombgauge}, which belongs to $L^3$ by Hardy--Littlewood--Sobolev, and then apply Lemma~\ref{gaugeinf} to pass from $\tilde A$ to $A'$ by changing the gauge.
	
	Note that by \eqref{eq:eq} we have $\nabla\cdot(|A_n'|A_n')=0$. Therefore, by Theorem \ref{strongconv}, we have that $A_n' \to A'$ in $L^q_\loc(\R^3,\R^3)$ for any $q<3$.
	
	On the other hand, since $(A_n')$ is bounded in $L^3$ by \eqref{eq:studyeq1}, a subsequence converges weakly in $L^3$ to some $A$. Because of the $L^q_\loc$ convergence to $A$ we conclude that $A=A'$. Applying this argument to a sub-subsequence of an arbitrary subsequence, we obtain the claimed weak convergence in $L^3$ of the full sequence.
\end{proof}


\section{Completion of the proof}

We are now in position to prove our main result, Theorem \ref{main}. Let $(A_n)\subset\mathcal Y$ be a minimizing sequence for $S$ with $\VERT A_n \VERT_3=1$. By Proposition \ref{nonzero}, after passing to a subsequence and after a translation and dilation, which we do not reflect in the notation, we have $\nabla\wedge A_n\rightharpoonup B$ in $L^{3/2}$ for some $B\not\equiv 0$. According to Proposition \ref{ekelandappl}, there is a sequence $(A_n')\subset\mathcal Y$ with $\nabla\cdot(|A_n'|A_n')=0$ and $\|A_n'\|_3 =1$ for all $n$ such that	
\begin{equation}
	\label{eq:endofproof0}
	\nabla\wedge A_n'- \nabla\wedge A_n \to 0
	\qquad\text{in}\ L^{3/2}(\R^3,\R^3)
\end{equation}
and
\begin{equation}
	\label{eq:eq1}
	\nabla\wedge  (|\nabla\wedge A_n'|^{-1/2} \nabla\wedge A_n') - S |A_n'| A_n'  = \nabla\wedge r_n
	\qquad\text{with}\ r_n \to 0 \ \text{in}\ L^3(\R^3,\R^3) \,.
\end{equation}

It follows from Lemmas \ref{goal22} and \ref{goal21} that there is an $A'\in\mathcal Y$ with $\nabla\wedge A'=B$ and $\nabla\cdot(|A'|A')=0$ such that $A_n'\rightharpoonup A'$ in $L^3(\R^3,\R^3)$ and such that for all $F\in C^1_c(\R^3)$,
\begin{equation}
	\label{eq:goal11}
	\int_{\R^3} |A_n'| A_n'\cdot F\,dx \to \int_{\R^3} |A'| A'\cdot F\,dx
\end{equation}
and
\begin{equation}
	\label{eq:goal12}
	\int_{\R^3} |\nabla\wedge A_n'|^{-1/2}(\nabla\wedge A_n')\cdot(\nabla\wedge F)\,dx \to \int_{\R^3} |\nabla\wedge A'|^{-1/2}(\nabla\wedge A')\cdot(\nabla\wedge F)\,dx \,.
\end{equation}
Because of \eqref{eq:goal11} and \eqref{eq:goal12} we deduce from \eqref{eq:eq1} that
\begin{equation}
	\label{eq:goal}
	\nabla\wedge ( |\nabla\wedge A'|^{-1/2}(\nabla\wedge A')) - S |A'|A' = 0 \,.
\end{equation}
Testing \eqref{eq:goal} with $A'$, we obtain
$$
\|\nabla\wedge A'\|_{3/2}^{3/2} - S \|A'\|_3^3 = 0
$$
and therefore
\begin{equation}
	\label{eq:endofproof}
	\frac{\|\nabla\wedge A'\|_{3/2}^{3/2}}{\|A'\|_3^{3/2}} = S \|A'\|_3^{3/2} \leq S
\end{equation}
since, by weak convergence, $\|A'\|_3 \leq \liminf_{n\to\infty} \|A_n'\|_3 = 1$. Note also that $A'\not\equiv 0$ since $B=\nabla\wedge A'\not\equiv 0$. By definition of $S$, \eqref{eq:endofproof} implies that $A'$ is a minimizer for $S$ and that
$$
\|A'\|_3 = 1
\qquad\text{and}\qquad
\|\nabla\wedge A'\|_{3/2} = S^{3/2} \,.
$$
Equality in the lower semicontinuity inequalities
$$
\|A'\|_3 \leq \lim_{n\to\infty} \|A_n'\|_3 = 1
\qquad\text{and}\qquad
\|\nabla\wedge A\|_{3/2} \leq \lim_{n\to\infty} \|\nabla\wedge A_n'\|_{3/2} = S^{3/2}
$$
implies, by \cite[Theorem 2.11]{LiLo}, that $(A_n')$ and $(\nabla\wedge A_n')$ converge strongly to $A'$ in $L^3$ and to $\nabla\wedge A'$ in $L^{3/2}$, respectively. 

Let us now pass from the sequence $(A_n')$ to the original sequence $(A_n)$. Because of \eqref{eq:endofproof0} we also have that $\nabla\wedge A_n\to \nabla\wedge A'$ strongly in $L^{3/2}$.

Let us assume that the gauge of the $A_n$ was fixed as in Lemma \ref{gaugeinf} by requiring $\nabla\cdot(|A_n|A_n)=0$. Of course, this condition is preserved under the translations and dilations that are performed in the above proof. By Theorem \ref{strongconv}, $A_n \to A'$ in $L^q_\loc(\R^3,\R^3)$ for any $q<3$. This, together with the boundedness of $A_n$ in $L^3$, implies by the same argument as in the proof of Lemma \ref{goal21} that $A_n \rightharpoonup A'$ in $L^3$. From $\|A_n\|_3=1$ and $\|A'\|_3=1$ we deduce as before that $A_n\to A'$ strongly in $L^3$. This concludes the proof of Theorem~\ref{main}.


\section{Proof of Theorem \ref{optimalb}}

We turn now to the proof of our second main result, Theorem \ref{optimalb}, and recall that the minimization problem $\Sigma$ was defined before that theorem. Let $(A_n)\subset\mathcal Y$ be a minimizing sequence for $\Sigma$, normalized such that $\nabla\cdot(|A_n|A_n)=0$, and let $\psi_n\in L^3(\R^3,\C^2)$ be a corresponding sequence such that $\sigma\cdot(-i\nabla-A_n)\psi_n=0$. By homogeneity we may assume that $\|\psi_n\|_3 = 1$. Since $(A_n)$ is a minimizing sequence, $\|\nabla\wedge A_n\|_{3/2}$ is bounded and therefore, by the Sobolev inequality (Lemma \ref{spositive}), $\|A_n\|_3 = \VERT A_n \VERT_3 \lesssim \|\nabla\wedge A_n\|_{3/2} \lesssim 1$. Moreover, by the zero mode equation, we have
$$
\| \sigma\cdot(-i\nabla)\psi_n \|_{3/2} = \| \sigma\cdot A_n \psi_n \|_{3/2} = \| |A_n|\psi_n \|_{3/2} \leq \| A_n \|_3 \|\psi_n\|_{3} \lesssim 1 \,.
$$
Applying the improved inequality \eqref{eq:improvedspinor} to our sequence, we deduce that
$$
\sup_{t>0} t \| e^{t\Delta} \sigma\cdot(-i\nabla)\psi_n \|_\infty \gtrsim 1 \,.
$$
Thus, by the same argument as in Proposition \ref{nonzero}, after translations and dilations, we can pass to a subsequence that converges weakly in $\dot W^{1,3/2}(\R^3,\C^2)$ to a limit $\tilde\psi\not\equiv 0$. By the usual Rellich--Kondrachov theorem we infer that it converges strongly in $L^q_\loc(\R^3,\C^2)$ for any $q<3$.

We translate and rescale the $A_n$ accordingly and note that these operations preserve the zero mode equation. By passing to another subsequence, we can ensure that $(\nabla\wedge A_n)$ converges weakly in $L^{3/2}(\R^3,\R^3)$ to some $\tilde B$. By the same argument as in the proof of Lemma \ref{goal21} we deduce that there is an $\tilde A\in\mathcal Y$ with $\nabla\wedge \tilde A = \tilde B$ and $\nabla\cdot(|\tilde A|\tilde A)=0$ such that $A_n\rightharpoonup \tilde A$ in $L^3(\R^3,\R^3)$ and $A_n \to \tilde A$ in $L^q_\loc(\R^3,\R^3)$ for any $q<3$. This step uses our nonlinear Rellich--Kondrachov theorem.

Consequently, $\sigma\cdot A_n\psi_n \to \sigma\cdot \tilde A\tilde\psi$ in $L^q_\loc(\R^3,\R^3)$ for any $q<3/2$. This allows us to pass to the limit in the distributional formulation of the zero mode equation and to conclude that $\sigma\cdot(-i\nabla-\tilde A)\tilde\psi =0$.

By weak convergence, one has
\begin{equation}
	\label{eq:liminfsigma}
	\| \nabla\wedge \tilde A\|_{3/2} \leq \liminf_{n\to\infty} \|\nabla\wedge A_n\|_{3/2} = \Sigma \,.
\end{equation}
By definition of $\Sigma$ and the fact that $\tilde A$ admits a zero mode $\tilde\psi\not\equiv 0$, we deduce that equality holds in \eqref{eq:liminfsigma}. This means that $\tilde A$ is a minimizer and that the convergence of $(\nabla\wedge A_n)$ to $\nabla\wedge \tilde A$ is strong in $L^{3/2}(\R^3,\R^3)$. This completes the proof of Theorem~\ref{optimalb}.
\qed

\bibliographystyle{amsalpha}

\end{document}